\documentclass[12pt]{article}
\usepackage{amssymb}
\usepackage{eufrak}
\usepackage{amsmath}
\usepackage{soul}
\usepackage{xr}
\usepackage{amsthm}
\include{xy}
\xyoption{all}

\begin{document}
\newtheorem{theorem}{Theorem}[section]
\newtheorem{lemma}[theorem]{Lemma}
\newtheorem{il}[theorem]{Illustration}
\newtheorem{coro}[theorem]{Corollary}
\newtheorem{prop}[theorem]{Proposition}
\newtheorem{defin}[theorem]{Definition}
\newtheorem{remark}[theorem]{Remark}
\newtheorem{remarks}[theorem]{Remarks}
\newtheorem{ex}[theorem]{Example}
\newtheorem{hyp}[theorem]{Hypothesis}
\newtheorem{Not}[theorem]{Notation}
\newtheorem{Conj}[theorem]{Conjecture}

\begin{titlepage}
\begin{center}
\vspace{2.8cm} {\LARGE\bf Iterative $q$-Difference Galois Theory   }
\end{center}

\begin{center}
  {\large{{\bf  Charlotte
Hardouin}}}\\\vspace{.5cm} {\large{{\bf IWR\\
Im Neuenheimer Feld 368\\
D-69120 Heidelberg
}}}
\end{center}

\begin{center} 4 November 2008 \\
\end{center}
\vspace{.9cm}
\end{titlepage}

\vfil \eject \tableofcontents

\vfil \eject

\section{Introduction}

Initially, the Galois theory of $q$-difference equations was built
for $|q|$ not equal to a root of unity (see for instance \cite{VPS2}).
This choice was made in order to avoid the increase of the field of
constants to a transcendental field. However, P.A. Hendricks studied
this problem in his PhD work under the supervision of M. van der Put
(see \cite{hen}). In Chapter $6$, he gave a notion of Galois groups
for $q$-difference equations over $\mathbb{C}(z)$ with $q^m=1$. His
idea was to compare the category $Diff_{\mathbb{C}(z)}$ of
$q$-difference modules over $\mathbb{C}(z)$ with the category
$FMod_{Z}$ of modules over the ring $\mathbb{C}(z^m)[t,t^{-1}]$. He
thus obtained an equivalence of categories and a fiber functor from
$Diff_{\mathbb{C}(z)}$ with values in the category
$Vect_{\mathbb{C}(z^m)}$ of vector spaces of finite dimension over
$\mathbb{C}(z^m)$. However, in his case there is no unique
Picard-Vessiot ring of a $q$-difference equation. This construction
is also not totally satisfying because we do not
want to have such transcendental base fields for Galois groups.\\

In the same matter, the question of the constant field for
differential modules in positive characteristic has given rise to
the construction of a  differential Galois theory in positive
characteristic. The first work in this direction was made by H.
Hasse and F.K. Schmidt \cite{Hass}, but it was only in 2000 when
B.H. Matzat and M. van der Put set up a modern and systematic
approach to this theory (see \cite{MP} and \cite{M}). The main idea
is to consider not only one derivation but a whole family of
derivations, called \textit{higher derivations} or \textit{iterative
derivations}. By defining the constants as the elements annihilated
by  the whole family of derivations, they succeeded in getting a
good constant field, for instance $\overline{\mathbb{F}_p}$ instead
of $\overline{\mathbb{F}_p}(z^p)$. So they were able to give a
complete description of the Picard-Vessiot theory of differential
equations
in positive characteristic and relate it to a Tannakian approach.\\

For $q$-difference theory, the problem is not the characteristic
 but the roots of unity. Inspired by the work of B.H.
Matzat and M. van der Put,  we  consider in this paper  a family of
\textit{iterative difference operators} instead of considering just
one difference operator, and in this way we stop the increase of the
constant field and succeed in setting up a Picard-Vessiot theory for
$q$-difference equations where $q$ is a root of unity. The theory we
obtain is quite the exact translation of the iterative differential
Galois theory developed by B.H. Matzat and M. van der Put to the
$q$-difference world. This analogy between iterative differential
Galois theory and iterative difference Galois theory could  perhaps
be explained in a more theoretical way,  as it is done in the paper
of Y. Andr\'{e} \cite{And} for classical theories. However we give some tracks of connections in section $3$. \\

The  interests of building such a theory are multiple. The first one
is to fill in  the gap in the classical $q$-difference Galois theory
 for $q$ a root of unity. Thus the theory of iterative $q$-difference
operators developed in this paper encompasses and extends the work
of Singer and van der Put (\cite{VPS2}). But this theory could also
provide a "good" functor of  confluence over complex fields from the world of
$q$-difference to the world of differential equations as it is done over $p$-adic fields
by A. Pulita (\cite{Pul}). Moreover it would be really interesting to establish a link between the
$(\sigma_q,
\partial)$-modules introduced by A. Pulita at the roots of unity and the iterative $q$-difference modules. In a similar way, it will be very enlightening to build   a
confluence functor in characteristic $p$ from iterative
$q$-difference modules to iterative differential modules.\\
  Another
goal of this theory will be to obtain an iterative $q$-difference
version of the Groethendieck Conjecture following the work of L. Di
Vizio \cite{Lu} and the work of P.A. Hendricks \cite{hen}. In other
words, we want to prove that the behavior of an
iterative $q$-difference module defined over $\overline{\mathbb{Q}}$ is determined by the behavior of its reduction modulo $p$ for
 almost all prime $p$. One could try also as it is conjectured in the differential case by Matzat and van der Put (see
 \cite{MP} p.51) to relate the finiteness of the
 the Galois group of an usual $q$-difference module to the existence of an iterative $q$-difference structure for the
 reduction of the module modulo $p$ for almost all prime $p$.\\

For  the whole paper, we fix an  algebraically closed field $C$ and
$q \in C$ with $q \neq 1$. Let $F=C(t)$ denote the field of rational
functions over $C$ and   $\sigma_q$  the automorphism of
$F$ which associates to a function $f(t)$ the function $f(qt)$.\\

In the second section, we introduce the arithmetic basis of
iterative $q$-difference algebra. In this section we work in all
generality, i.e., we do not make any assumptions wether  $q$ is  a
root of unity or not. With this choice we want to emphasize the fact
that we just generalize the Galois theory of $q$-difference of M.F
Singer and M. van der Put
  (\cite{VPS2}).
From the third section until the end of the paper, we will restrict
ourselves to the case of $q$ a primitive root of unity, where the
most peculiar phenomena appear. In Section $3$ we define the
category of iterative $q$-difference modules and their relation with
some specific category of projective systems. As in  \cite{M}, the
equivalence of categories yields  a family of $q$-difference
equations, related to the fact that an iterative $q$-difference
operator is a family of maps. Such a family of equations can be
regarded in two different ways, a general  and a relative one using
the projective
system. Both formulations are used in  later sections.
We build  a Picard-Vessiot theory for iterative
$q$-difference equations by using the classical theory as formulated
for instance in \cite{VPS2}. \\
In Section $5$, we adopt Kolchin's way of thinking  and show how an
iterative $q$-difference Galois group is formed by the $C$-points of
an affine  group-scheme. We also obtain the analogue of Kolchin's
theorem for our theory and the usual Galois
correspondence. To be a little more concrete, at the
end of the section, we give a method to realize linear algebraic
groups of dimension one as iterative $q$-difference Galois groups.\\
As a conclusion to this paper, we state an  analogue of the
Grothendieck-Katz conjecture for  iterative $q$-difference Galois
groups  as in  the work of L. Di Vizio.\\

\noindent \textbf{Acknoledgements}\\

I would like to thank A. Roescheisen and J. Hartmann for all their
help, remarks and so useful comments and also L. Di Vizio specially for enlightening discussions about section $5$.  Last but not least, I am
sincerely grateful to the Professor B.H. Matzat for the inspiration
his theory has provided to me and for all his help and encouragement
to pursue this study. The author also thanks the referee for his important suggestions.

\section{Iterative $q$-difference rings}

In considering an element $q$ of a field $C$ which may be a
primitive root of unity and  trying to construct a $q$-difference
Galois theory, we have to deal with the problem that the field of
constants of the usual $q$-difference operator extends to a
transcendental field. To avoid this increase of the constants, we
have to consider a more arithmetic approach, such as the one
introduced by  H. Hasse and F.K. Schmidt \cite{Hass} for
differential equations in positive characteristic. Until the end of this article, we let
$F=C(t)$ denote the field of rational functions over an algebraically closed field $C$ and
$\sigma_q$ the $q$-difference operator of $F$ defined as follows :
$\sigma_q(f(t)):=f(qt)$.

\subsection{$q$-Arithmetic properties}
In this paragraph, we just recall the most usual $q$-arithmetical
objects.

\begin{defin}\label{defin:bin}
Let $k \geq \mathbb{N}^*$. Put   $ [0]_q=0, \ [k]_q:= \frac{q^k-1}{q-1}.$
\begin{enumerate}
 \item Let  $[k]_q!$ denote the element of $C$ defined by  $[k]_q [k-1]_q...[1]_q$ and  by convention set $[0]_q!=1$. We will say that $[k]_q!$ is  the $q$-factorial of
 $k$.
\item Let  $\binom{r}{k}_q$ denote the element of $C$ defined by
$\frac{[r]_q!}{[k]_q![(r-k)]_q!}$. We will say that $\binom{r}{k}_q$
is the $q$-binomial coefficient of $r$ over $k$.
\item $(t;q)_m:=(1-t)(1-qt)...(1-q^{m-1}t)$.
\end{enumerate}
\end{defin}

\begin{prop} \label{prop:binform}
\begin{enumerate}
\item $\binom{r}{0}_q= \binom{r}{r}_q =1$.
\item $\binom{0}{k}_q =0$ if $k \neq 0$ and $\binom{0}{0}_q=1$.

\item Assume that $q$ is a  primitive $n$-th root of unity. Then for  two integers
$a>b$,

 \begin{equation}\label{eqn:mulb}
 \binom{an}{bn}_q =\binom{a}{b}
 .\end{equation}
\item $\sum_{i+j=k, i\leq s, j\leq r}
\binom{r}{j}_q\binom{s}{i}_q q^{i(r-j)} = \binom{r+s}{k}_q$ for all
$(k,r,s) \in \mathbb{N}^3$ with $r+s\geq k$.
\end{enumerate}
\end{prop}
\begin{proof}
~~\\
\textbf{Proof of part $3$}\\
Let $m \in \mathbb{N}$. One expand the function $(t;q)_m$ of $C(t)$ defined in
\ref{defin:bin} part $3$, i.e.
\begin{equation}\label{eqn:fact}
(t;q)_m= \sum_{j=0}^m(-1)^j\binom{m}{j}_q q^{j(j-1)/2} t^j
.\end{equation} Because $q^n=1$ and $n$ is the order of $q$, we have
$ (t;q)_{an}=(t;q)_n^a$. Using Equation (\ref{eqn:fact}), we obtain
$$\sum_{j=0}^{an}(-1)^j\binom{an}{j}_q q^{j(j-1)/2} t^j=
\sum_{j=0}^a \binom{a}{j} (-1)^{nj}q^{n(n-1)j/2} t^{nj}.$$ By
comparing the terms in $t^{bn}$, we have $\binom{an}{bn}_q
q^{\frac{bn(bn-1)}{2}}=\binom{a}{b}
 q^{b\frac{n(n-1)}{2}}.$\\
\textbf{Proof of part $4$}\\
Let $(k,r,s) \in \mathbb{N}^3$ with $r+s\geq k$. We have
\begin{equation}\label{eqn:fact1}
(t;q)_{r+s}=(t;q)_r(q^rt;q)_s.
\end{equation}
By comparing the terms in $t^k$, we obtain
$$(-1)^kq^{k(k-1)/2}\binom{r+s}{k}_q=\sum_{i+j=k, i\leq s, j\leq r}
(-1)^{i+j}q^{k(k-1)/2}\binom{r}{j}_q\binom{s}{i}_q q^{i(r-j)} .$$
\end{proof}
\begin{remark}\label{remark:pnul}
If $C$ is of characteristic $p>0$, then for $p^j>i$ we get from
equation (\ref{eqn:mulb}) $\binom{np^j}{ni}_q=0$.
\end{remark}

\subsection{Iterative $q$-difference ring}

In this paragraph, we establish the formal properties of the iterative
$q$-difference operator. In the world of $q$-difference the analogue
of the derivation $\frac{d}{dt}$ is the operator
$\delta_q:=\frac{\sigma_q -id}{(q-1)t}$ (see for instance
\cite{And2} p.1). Heuristically speaking, when $q$ goes to $1$,
$\delta_q$ goes to the usual derivation $\frac{d}{dt}$. Thus the
main idea of our constructions is to deform the iterative
derivations into iterative difference operators by replacing
$\frac{d}{dt}$ by $\delta_q$ and all the arithmetical factors occurring
in their Definition $1.1$ of \cite{M} by their $q$-analogues. The only
change appears at the part $4$ of Definition \ref{defin:qder}, where
a twist by $\sigma_q$ occurs.
\begin{defin}\label{defin:qder}
Let $R$ be a finitely generated  $C(t)$-algebra having an automorphism also called $\sigma_q$ extending $\sigma_q$ on $C(t)$ (see
\cite{VPS2}, section $1.1$) and let $\delta_R^{*}:= (\delta_R^{(k)})_{k \in \mathbb{N}}$ be a
collection of maps from $R$ to $R$.  The family $\delta_R^{*}$ is
called an \textbf{iterative $q$-difference operator} on $R$, if for
all $a,b \in R$ and all $i,j,k \in \mathbb{N}$, the following
properties are satisfied
\begin{enumerate}
\item $\delta_R^{(0)}=id$,
\item $\delta_R^{(1)} =\frac{\sigma_q -id}{(q-1)t}$,
\item $\delta_R^{(k)}(x +y)= \delta_R^{(k)}(x) + \delta_R^{(k)}(y)$,
\item $\delta^{(k)}(ab) = \sum_{i+j=k} \sigma_q^i(\delta_R^{(j)}(a))
\delta_R^{(i)}(b)$,
\item $\delta_R^{(i)} \circ \delta_R^{(j)}=\binom{i+j}{i}_q
\delta_R^{(i+j)}$.

\end{enumerate}
 The set of
iterative $q$-difference operators is denoted by $ID_q(R)$. For
$\delta_R^* \in ID_q(R)$, the tuple $(R,\delta_R^*)$ is called an
\textbf{iterative $q$-difference ring} ($ID_q$-ring). We say that an
element $c$ of $R$ is a \textbf{constant} if $\delta_R^{(k)}(c)=0$
 for all $k \in
\mathbb{N}^*$. We  denote by $C(R)$ the ring of constants of $R$.

\end{defin}

\begin{remark}\label{remark:ech}

 If $R$ is a ring, then $C(R)$ is a ring. If $R$ is a
field, then $C(R)$ is a field.

\end{remark}

\begin{lemma}\label{lemma: ech}
For all $j,i \in \mathbb{N}$, we have
\begin{equation}\label{eqn:sd}\sigma_q^j \delta_R^{(i)}=
\frac{1}{q^{ji}}\delta_R^{(i)} \sigma_q^j .\end{equation}
\end{lemma}
\begin{proof}

In order to prove Equation (\ref{eqn:sd}), it is sufficient to prove
it for $j=1$, the general case  obviously follows from this
case.\\
For all $k >0$, we have
\begin{equation}\label{eqn:nul}\delta_R^{(k)}(t
\frac{1}{t})=0=\delta_R^{(k)}(t^{-1})t +
\sigma_q(\delta_R^{(k-1)}(t^{-1})).\end{equation}

By part $5$ of Definition  \ref{defin:qder}, we get $\delta_R^{(1)}
\circ \delta_R^{(i)} =\delta_R^{(i)} \circ \delta_R^{(1)}$ for all
$i \in \mathbb{N}$. Using part $2$ and $4$, we obtain that

\begin{equation} \label{eqn:nul1} \frac{\sigma_q -id}{t} \circ \delta_R^{(i)}(x)=\delta_R^{(i)}
\circ(\frac{\sigma_q
-id}{t})(x)=\sum_{k=1}^i\sigma_q^k(\delta_R^{(i-k)}(\sigma_q
-id)(x))\delta_R^{(k)}(t^{-1}) + \delta_R^{(i)}((\sigma_q -id)(x))
t^{-1}\end{equation} for all $x \in R$ and $i \in \mathbb{N}$. By  Equation
(\ref{eqn:nul1}), we get
$$ \frac{\sigma_q}{t} \circ \delta_R^{(i)}(x)=\frac{-1}{t} \sigma_q
[\sum_{k=0}^{i-1}\sigma_q^{k}(\delta_R^{(i-1-k)}(\sigma_q
-id)(x))\delta_R^{(k)}(t^{-1})] + \frac{\delta_R^{(i)} \circ
\sigma_q(x)}{t},$$ i.e.,
$$ \frac{\sigma_q}{t} \circ \delta_R^{(i)}(x)= \frac{-1}{t} \sigma_q(\delta_R^{(i-1)}\circ
\delta_R^{(1)}(x)) + \frac{\delta_R^{(i)} \circ \sigma_q(x)}{t} $$
that is,

$$\frac{\sigma_q}{t} \circ \delta_R^{(i)}(x)= -\frac{q-1}{t} \sigma_q
\circ( \frac{q^i-1}{q-1} \delta_R^{(i)} )(x) + \frac{\delta_R^{(i)}
\circ \sigma_q(x)}{t}.$$ This last equation gives $$\sigma_q
\delta_R^{(i)}(x)= \frac{1}{q^i}\delta_R^{(i)} \sigma_q (x)$$ which
concludes the proof.
\end{proof}

\begin{remark}[Classical case] If $q$ is not a root of unity then $\delta_R^{(k)}=\frac{{(\delta_R^{(1)})}^k}{[k]_q!}$
and the iterative $q$-difference rings that we consider are the
$q$-difference algebras extensions of $C(t)$ studied by M. van der Put and M.F Singer in
\cite{VPS2} chapter $1$.
\end{remark}

\textbf{Main example : The  field of rational functions over $C$}\\

\begin{defin}\label{defin:itrat}

Let $k \in \mathbb{N}$. Let $\delta_q^{(k)}$ denote the additive map
from $C[t]$ to $C[t]$ defined by    $ \delta_q^{(k)}(\lambda
t^{r}):=\lambda \binom{r}{k}_q t^{r-k}, $ for all $r \in
\mathbb{N}$, and $\lambda \in C $. Using the formula
$\delta^{(k)}(ab) = \sum_{i+j=k} \sigma_q^i(\delta_q^{(j)}(a))
\delta_q^{(i)}(b)$, we extend $\delta_q^{(k)}$ to $F=C(t)$.
\end{defin}

\begin{prop}\label{prop:qder}
The collection  $(\delta_q^{(k)})_{k \in \mathbb{N}}$ of maps from
$F$ to $F$, defined previously, satisfy

\begin{enumerate}
\item $\delta_q^{(0)}=id$,
\item $\delta_q^{(1)}= \frac{\sigma_q -id}{(q-1)t}$,
\item for all $k \in \mathbb{N}$, the map $\delta_q^{(k)}$is
additive,
\item $\delta_q^{(i)} \circ \delta_q^{(j)}=\binom{i+j}{i}_q
\delta_q^{(i+j)}$.

\end{enumerate}
\end{prop}
\begin{proof}

By construction of $(\delta_q^{(k)})_{k \in \mathbb{N}}$, it is
sufficient to prove that all the formulas hold upon evaluation on
$t^r$ with $r \in \mathbb{N}$.

\begin{enumerate}
\item Because $\binom{k}{0}_q=1$, it is obvious that
$\delta_q^{(0)} =id$.
\item For all $r \in \mathbb{N}$, we have
$\delta_q^{(1)}(t^r)=\binom{r}{1}_q t^{r-1}=\frac{q^rt^r
-t^r}{(q-1)t}= \frac{\sigma_q -id}{(q-1)t}(t^r)$.
\item
Let $r \in \mathbb{N}$. We have
$$\delta_q^{(i)} \circ \delta_q^{(j)}(t^r)=
\binom{r-j}{i}_q \binom{r}{j}_q t^{r}$$ and
$$\binom{r-j}{i}_q \binom{r}{j}_q=\binom{i+j}{i}_q
\binom{r}{i+j}_q,$$  which gives $$\delta_q^{(i)} \circ
\delta_q^{(j)}(t^r)=\binom{i+j}{i}_q \delta_q^{(i+j)}(t^r).$$
\end{enumerate}
\end{proof}
\begin{prop}
The field  $F=C(t)$ endowed with the collection of maps
$(\delta_q^{(k)})_{k \in \mathbb{N}}$ as in Definition
\ref{defin:itrat} is an iterative $q$-difference field with
$\delta_q^{(n)}(t^n)=1$ for all $n \in \mathbb{N}$ and thus
$C(F)=C$.
\end{prop}

\textbf{Tensor product of $ID_q$-rings}\\

\begin{lemma}\label{lemma:com}
Let $(R_1,\delta_{R_1}^*)$ and $(R_2,\delta_{R_2}^*)$ be two
iterative $q$-difference rings. We have
\begin{equation}\label{eqn:com}\sum_{i+j=k}\sigma_q^i(\delta_{R_1}^{(j)}(a))\otimes
\delta_{R_2}^{(i)}(b)= \sum_{i+j=k} \delta_{R_1}^{(j)}(a)\otimes
\sigma_q^j( \delta_{R_2}^{(i)}(b))\end{equation} for all $k \in
\mathbb{N}$, $(a,b) \in R_1 \times R_2$.
\end{lemma}

\begin{proof}
The formula (\ref{eqn:com}) is obviously true for $k=1$, using the
definition of $\delta^{(1)}$. If (\ref{eqn:com}) holds for $k$ and
$l$ in $\mathbb{N}$, we have

$$\binom{k+l}{k}_q \sum_{i+j=k+l}\sigma_q^i(\delta_{R_1}^{(j)}(a))\otimes
\delta_{R_2}^{(i)}(b)=\left(\sum_{r+s=l}\sigma_q^r(\delta_{R_1}^{(s)})\otimes
\delta_{R_2}^{(r)}\right)\left(\sum_{i+j=k}\sigma_q^i(\delta_{R_1}^{(j)}(a))\otimes
\delta_{R_2}^{(i)}(b)\right)$$ that is $$
(\sum_{r+s=l}\sigma_q^r(\delta_{R_1}^{(s)})\otimes
\delta_{R_2}^{(r)})\left(\sum_{i+j=k}\sigma_q^i(\delta_{R_1}^{(j)}(a))\otimes
\delta_{R_2}^{(i)}(b)\right)= \left(\sum_{r+s=l}
\delta_{R_1}^{(r)}\otimes \sigma_q^r(
\delta_{R_2}^{(s)})\right)\left( \sum_{i+j=k}
\delta_{R_1}^{(j)}(a)\otimes \sigma_q^j(
\delta_{R_2}^{(i)}(b))\right)$$ and thus
$$\binom{k+l}{k}_q \sum_{i+j=k+l}\sigma_q^i(\delta_{R_1}^{(j)}(a))\otimes
\delta_{R_2}^{(i)}(b)=\binom{k+l}{k}_q \sum_{i+j=k+l}
\delta_{R_1}^{(j)}(a)\otimes \sigma_q^j( \delta_{R_2}^{(i)}(b)).$$
Then, if $\binom{k+l}{k}_q \neq 0$, the formula (\ref{eqn:com})
holds for $k+l$. If $q$ is not a root of unity, we can conclude by
induction.\\
 Assume now that $q^n=1$. It remains to show that
Formula (\ref{eqn:com}) holds for $k \in
n\mathbb{N}$. We will first prove it for $k=n$.\\
Because $\sum_{i+j=n}\sigma_q^i(\delta_{R_1}^{(j)}(a))\otimes
\delta_{R_2}^{(i)}(b)=\delta_{R_1}^{(n)}(a) \otimes b + a \otimes
\delta_{R_2}^{(n)}(b) +
\sum_{i=1}^{n-1}\sigma_q^i(\delta_{R_1}^{(n-i)}(a))\otimes
\delta_{R_2}^{(i)}(b)$, the proof for $k=n$ will be complete if we
show that
\begin{equation}\label{eqn:com1}\sum_{i=1}^{n-1}\sigma_q^i(\delta_{R_1}^{(n-i)}(a))\otimes
\delta_{R_2}^{(i)}(b) = \sum_{i=1}^{n-1}\delta_{R_1}^{(i)}(a)\otimes
\sigma_q^i(\delta_{R_2}^{(n-i)}(b)).\end{equation} We have
$\delta^{(k)}=\frac{(\delta^{(1)})^k}{[k]_q!}$ and
$$(\delta^{(1)})^k= \frac{(-1)^k}{((q-1)t)^k}\sum_{j=0}^k
(-1)^j\binom{k}{j}_{q^{-1}}q^{-\frac{j(j-1)}{2}}\sigma_q^j=\frac{1}{((q-1)t)^k}\sum_{j=0}^k
a_{j,k} \sigma_q^j$$ for $0<k<n$ (see \cite{Lu}, Lemma $1.1.10$).
Then,
\begin{equation} \label{eqn:com2}\sum_{i=1}^{n-1}\sigma_q^i(\delta_{R_1}^{(n-i)}(a))\otimes
\delta_{R_2}^{(i)}(b) = \frac{1}{((q-1)t)^n} \sum_{l=1}^n
\sum_{k=0}^l  \sigma_q^l(a) \otimes \sigma_q^k(b) (\sum_{i=k, i \neq
0}^{i=l, i\neq n}
\frac{a_{l-i,n-i}a_{k,i}q^{-i(n-i)}}{[n-i]_q![i]_q!}).
\end{equation}
If $l\neq n$, $k \neq 0$ and $l \neq k$, we have
$$\sum_{i=k, i \neq
0}^{i=l, i\neq n}
\frac{a_{l-i,n-i}a_{k,i}q^{-i(n-i)}}{[n-i]_q![i]_q!} =
\frac{(-1)^{l+n}q^{-\frac{n(n-1)}{2}}}{[n-l]_{q^{-1}}![k]_{q}![l-k]_{q}!}
\sum_{i=0}^{l-k} (-1)^i \binom{l-k}{i}_q q^{\frac{i(i-1)}{2}} =0$$
(expand $(1;q)_{l-k}$). If $l=n$, then
$$\sum_{i=k, i \neq
0}^{i=n, i\neq n}
\frac{a_{n-i,n-i}a_{k,i}q^{-i(n-i)}}{[n-i]_q![i]_q!} =
\frac{(-1)^{n+k+1}q^{-\frac{n(n-1)}{2}}}{[k]_q![n-k]_{q^{-1}}}=
\sum_{i=0, i \neq 0}^{i=k, i\neq n}
\frac{a_{k-i,n-i}a_{0,i}q^{-i(n-i)}}{[n-i]_q![i]_q!}$$ (expand
$(1;q)_k$ and $(1;q)_{n-k}$). Because $\sigma_q^n=id$ , it follows
that the equation (\ref{eqn:com2}) is symmetric in $a$ and $b$. Thus
the formula (\ref{eqn:com1}) holds
and the equation (\ref{eqn:com}) is true for $k=n$.\\
For $k=2n$, we have
$$\sum_{i+j=n}\sigma_q^i(\delta_{R_1}^{(j)}(a))\otimes
\delta_{R_2}^{(i)}(b)=\delta_{R_1}^{(2n)}(a) \otimes b + a \otimes
\delta_{R_2}^{(2n)}(b) +$$ $$
\sum_{i=1}^{n-1}\sigma_q^i(\delta_{R_1}^{(2n-i)}(a))\otimes
\delta_{R_2}^{(i)}(b)+\delta_{R_1}^{(n)}(a)
\otimes\delta_{R_2}^{(n)}(b) +
\sum_{i=n+1}^{2n-1}\sigma_q^i(\delta_{R_1}^{(2n-i)}(a))\otimes
\delta_{R_2}^{(i)}(b).$$ Because $\delta^{(2n-i)}
=\delta^{(n-i)}\circ \delta^{(n)}$ for all $i=1,...,n-1$, we obtain
by (\ref{eqn:com1})
$$\sum_{i=1}^{n-1}\sigma_q^i(\delta_{R_1}^{(2n-i)}(a))\otimes
\delta_{R_2}^{(i)}(b)=\sum_{i=1}^{n-1}\sigma_q^i(\delta_{R_1}^{(
n-i)}( \delta_{R_1}^{(n)}(a)))\otimes
\delta_{R_2}^{(i)}(b)=\sum_{i=1}^{n-1} \delta_{R_1}^{(n+i)}(a)
\otimes \sigma_q^i(\delta_{R_2}^{(n-i)}(b))$$
$$=\sum_{i=n+1}^{2n-1}\delta_{R_1}^{(i)}(a) \otimes
\sigma_q^i(\delta_{R_2}^{(2n-i)}(b)).$$ We also have
$$\sum_{i=n+1}^{2n-1}\sigma_q^i(\delta_{R_1}^{(2n-i)}(a))\otimes
\delta_{R_2}^{(i)}(b)= \sum_{i=1}^{n-1}\delta_{R_1}^{(i)}(a) \otimes
\sigma_q^i(\delta_{R_2}^{(2n-i)}(b)).$$ This concludes the proof for
$k=2n$. The same arguments gives the other cases.
\end{proof}

\begin{prop}[Definition]\label{prop:tens}
Let $(R_1,\delta_{R_1}^*)$ and $(R_2,\delta_{R_2}^*)$ be two
iterative $q$-difference rings. We define a collection of maps
$(\delta_{R_1 \otimes R_2}^{(k)})_{k \in \mathbb{N}}$  from $R_1
\otimes_F R_2$ to $R_1 \otimes_F R_2$ as follows :
$$\delta_{R_1\otimes R_2}^{(k)}(r_1 \otimes r_2):=
\sum_{i+j=k}\sigma_q^i(\delta_{R_1}^{(j)}(r_1))\otimes
\delta_{R_2}^{(i)}(r_2) \ \mbox{ for all} \ k \in \mathbb{N}, r_1
\in R_1 \ \mbox{and} \ r_2 \in R_2.$$ Then $(R_1 \otimes_F
R_2,\delta_{R_1 \otimes R_2}^*)$ is an iterative $q$-difference
ring.
\end{prop}
\begin{proof}
 It is obvious that the family $(\delta_{R_1\otimes
R_2}^{(k)})_{k \in \mathbb{N}}$ satisfies the three first parts of
Definition \ref{defin:qder}. By Lemma \ref{lemma:com} we have
$$\delta_{R_1\otimes R_2}^{(k)}(r_1 \otimes r_2) =
\sum_{i+j=k}\sigma_q^i(\delta_{R_1}^{(j)}(r_1))\otimes
\delta_{R_2}^{(i)}(r_2)= \sum_{i+j=k} \delta_{R_1}^{(j)}(r_1)\otimes
\sigma_q^j( \delta_{R_2}^{(i)}(r_2))$$ for all $k \in \mathbb{N}$.
 Let $(a,c) \in R_1^2$ and $(b,d) \in
R_2^2$. We have
$$ \delta_{R_1 \otimes R_2}^{(k)}((a\otimes b)(c \otimes d))=
\sum_{i+j=k} \sigma_q^{i}(\delta_{R_1}^{(j)}(ac)) \otimes
\delta_{R_2}^{(j)}(bd),$$

 $$\delta_{R_1 \otimes R_2}^{(k)}((a\otimes b)(c \otimes
 d))=\sum_{i_1+i_2+j_1+j_2=k}\sigma_q^{i_1+i_2+j_1}(\delta_{R_1}^{(j_2)}(a))\sigma_q^{i_1+i_2}(\delta_{R_1}^{(j_1)}(c))
 \otimes
 \sigma_q^{i_1}(\delta_{R_2}^{(i_2)}(b))\delta_{R_2}^{(i_2)}(d),$$
 and thus,
 $$\delta_{R_1 \otimes R_2}^{(k)}((a\otimes b)(c \otimes
 d))=\sum_{i_1+j_2+i=k} \sigma_q^{i_1 + i}(\delta_{R_1}^{(j_2)}(a))
 \otimes \delta_{R_2}^{(i_1)}(d) ( \sigma_q^{i_1}(\delta_{R_1
 \otimes R_2}^{(i)}(c \otimes b))).$$
This gives
$$\delta_{R_1 \otimes R_2}^{(k)}((a\otimes b)(c \otimes
 d))=\sum_{i_1+i_2+j_1+j_2=k}\sigma_q^{i_1+i_2+j_1}(\delta_{R_1}^{(j_2)}(a))\sigma_q^{i_1}(\delta_{R_1}^{(i_2)}(c))
 \otimes
 \sigma_q^{i_1+i_2}(\delta_{R_2}^{(j_1)}(b))\delta_{R_2}^{(i_1)}(d),$$
and thus
$$\delta_{R_1 \otimes R_2}^{(k)}((a\otimes b)(c \otimes
 d))= \sum_{i+j=k} \sigma_q^i(\delta_{R_1 \otimes R_2}^{(j)}(a \otimes
 b ) )\delta_{R_1 \otimes R_2}^{(i)}(c \otimes d).$$
 This is part $4$ of Definition \ref{defin:qder}.\\
 We now prove part $5$. Let $(k,l) \in \mathbb{N}^2$ and $(a,b) \in
 R_1 \times R_2$. We have
 $$\delta_{R_1 \otimes R_2}^{(k)} \circ \delta_{R_1 \otimes
 R_2}^{(l)}(a \otimes b)= \sum_{i+j=l, i_1+j_1=k}q^{ij_1}
 \binom{j_1+j}{j_1}_q
 \binom{i_1+i}{i}_q\sigma_q^{i_1+i}(\delta_{R_1}^{(j_1+j)}(a))
 \otimes \delta_{R_2}^{(i_1+i)}(b),$$
that is
$$\delta_{R_1 \otimes R_2}^{(k)} \circ \delta_{R_1 \otimes
 R_2}^{(l)}(a \otimes b)= \sum_{r+s=k+l}
 \sigma_q^r(\delta_{R_1}^{(s)}(a)) \otimes \delta_{R_2}^{(r)}(b) (\sum_{i+j=k, i\leq s, j\leq r}
\binom{r}{j}_q\binom{s}{i}_q q^{i(r-j)}).$$ Using part $5$ of
Proposition \ref{prop:binform}, we obtain
$$\delta_{R_1 \otimes R_2}^{(k)} \circ \delta_{R_1 \otimes
 R_2}^{(l)}(a \otimes b)=  \binom{k+l}{k}_q\sum_{r+s=k+l}
 \sigma_q^r(\delta_{R_1}^{(s)}(a)) \otimes \delta_{R_2}^{(r)}(b),$$
 that is $$\delta_{R_1 \otimes R_2}^{(k)} \circ \delta_{R_1 \otimes
 R_2}^{(l)}(a \otimes b)=  \binom{k+l}{k}_q \delta_{R_1 \otimes
 R_2}^{(k+l)}(a \otimes b).$$
\end{proof}

\subsection{Twisted ring of formal power series}

This paragraph is devoted to the relations between $ID_q$-rings and
rings of formal power series. By encoding all properties of an
iterative $q$-difference operator into twisted formal power series,
Property \ref{prop:tay}   provides us with a very powerful tool for
the proofs to come.\\
This kind of twisted ring appears already in the work of Yves Andr\'{e} (see \cite{And} $1.4.2.1$). % anneau

\begin{defin}\label{defin:tw}
Let $(R,\delta_R^*)$ be an iterative $q$-difference ring. The
twisted ring $R^{\sigma_q}[[T]]$ of formal series with coefficients
in $R$ is defined as follows : the additive  structure of
$R^{\sigma_q}[[T]]$ is the same as the one of $R[[T]]$, the
multiplicative  structure is given by
$$ \lambda T^r * \mu T^k:= \sigma_q^r(\mu)\lambda T^{r+k}$$
and extended by distributivity to $R[[T]]$.
\end{defin}

We will denote by "." the usual multiplication law on $R[[T]]$.

\begin{lemma}
 The twisted ring
$(R^{\sigma_q}[[T]],+,* )$ as  in Definition \ref{defin:tw}  is a
non commutative  ring with unity.
\end{lemma}
\begin{proof}
We have :$$ \lambda T^r*1=\lambda T^r*T^0=\sigma_q^r(1)\lambda
T^{r+0}=\lambda T^r=1 *\lambda T^r=\sigma_q^0(\lambda )T^r=\lambda
T^r.$$ Thus $1$ is a neutral
element for the twisted multiplication $*$. \\
Let us prove then, that $*$ is  associative.
$$\nu T^s *(\lambda T^r * \mu T^k)=  \nu T^s*(\sigma_q^r(\mu)\lambda
T^{r+k})=\sigma_q^{r+s}(\mu)\sigma_q^s(\lambda)\nu T^{r+s+k}$$ and
$$(\nu T^s *\lambda T^r) * \mu T^k=(\sigma_q^s(\lambda)\nu
T^{r+s})* \mu T^k= \sigma_q^{r+s}(\mu)\sigma_q^{s}(\lambda) \nu
T^{r+s+k}$$ give
$$\nu T^s *(\lambda T^r * \mu T^k)=(\nu T^s *\lambda T^r) * \mu
T^k.$$ The product $*$ is therefore associative.\\
Now, we want to introduce  an iterative $q$-difference operator on
$(R^{\sigma_q}[[T]],+,.)$, that is to say, a collection of maps
$\delta_T^*$ which  satisfies all the properties of Definition
\ref{defin:qder}.\\
First we need an automorphism  $\sigma_q$ on
$(R^{\sigma_q}[[T]],+,.)$ such that $ (R^{\sigma_q}[[T]],+,.)$ is a
$q$-difference ring extension of $F$. We put $\sigma_q(aT^i):=
\sigma_q(a)q^iT^i$  for all $i \in \mathbb{N}$ and $ a \in R $. By
extending  this definition $R$-linearly, $R^{\sigma_q}[[T]]$ becomes
a $q$-difference ring extension
of $F$.\\
 We put
$ \delta_T^{(k)}(T^r):=\binom{r}{k}_qT^{r-k}$ for all $(k,r) \in
\mathbb{N}^2$  and extend this definition by $R$-linearity.
Obviously $(\delta_T^{(k)})_{k \in \mathbb{N}}$ is an iterative
$q$-difference operator over $(R^{\sigma_q}[[T]],+,.)$ (see
Definition
\ref{defin:qder}).\end {proof}

\begin{defin}
 For all $a \in R$, $$ \bold{T}_a(T):=\sum_{k \in \mathbb{N}} \delta_R^{(k)}(a)T^k .$$
 is called the $q$-\textbf{iterative Taylor
series} of $a$. We define the map $\bold{T} : R \rightarrow R^{\sigma_q}[[T]]$ where $\bold{T}(a):= \bold{T}_a(T)$.\\\

\end{defin}

\begin{prop}\label{prop:tay}

Let $R$ be a $q$-difference ring extension of $F$ and let
$\delta_R^*=(\delta_R^{(k)})_{ k \in \mathbb{N}}$ be a sequence of
maps from $R$ to $R$. Let $\delta_T^*$ be the iterative
$q$-difference operator of $(R^{\sigma_q}[[T]],+,.)$ defined
previously, and let $\bold{I}$ denote the map
$$\xymatrix{ \bold{I}: & R^{\sigma_q}[[T]] \ar[r] & R,
   & \sum_{k \in \mathbb{N}}a_k T^k \ar@{|->}[r] & a_0}.$$
Then  $\delta_R^*$ is an iterative $q$-difference  operator for $R$
if and only if
\begin{enumerate}
\item $\bold{T}$ is a ring homomorphism from $R$ to $ (R^{\sigma_q}[[T]],+,* )$     , with $\bold{I} \circ
\bold{T}=id_R$,
\item $\delta_T^{(k)}\circ \bold{T} =\bold{T} \circ \delta_R^{(k)}$
for all $k \in \mathbb{N}$.
\end{enumerate}
\end{prop}

\begin{proof}
The fact that $\bold{T}$ is additive is equivalent to statement $3$
in Definition \ref{defin:qder}. The compatibility of $\bold{T}$ with
the multiplication law in $R$ and the twisted law $*$ in
$R^{\sigma_q}[[T]]$, in the case where $\delta_R^*$ is an iterative
$q$-difference operator comes from
 the equations $$ \bold{T}_{ab}(T):=\sum_{k\in \mathbb{N}} \delta_R^{(k)}(ab)T^k
=
 \sum_{k\in \mathbb{N}}(\sum_{i+j=k} \sigma_q^i(\delta_R^{(j)}(a))
\delta_R^{(i)}(b))T^k = \bold{T}_{a}(T)*\bold{T}_{b}(T).$$ The
second  property is equivalent to the property $5$ of the same
definition.
\end{proof}
\subsection{Iterative $q$-difference morphisms and iterative
$q$-difference ideals}

\begin{defin}
Let $(R,\delta_R^*)$ and $(S,\delta_S^*)$ be two iterative
$q$-difference rings. We  say that a  ring morphism $\phi$ from $R$
to $S$ is an \textbf{iterative $q$-difference morphism} if and only
if $ \delta_S^{(k)} \circ \phi =\phi \circ \delta_R^{(k)}$ for all
$k \in \mathbb{N}$.
\end{defin}
The set of all iterative $q$-difference morphisms from $R$ to $S$ is
denoted by $Hom_{ID_q}(R,S)$.\\
An iterative $q$-difference ideal $I \subset R$ ($ID_q$-ideal) is an
ideal of $R$ stable by $\delta_R^{(k)}$ for all $k \in \mathbb{N}$.

\begin{lemma}\label{lemma:rad}

  Let $I$ be an
$ID_q$-ideal of an iterative $q$-difference ring $R$, that is to say
that $I$ is stable under the action of $\delta_R^*$. Then the
radical of $I$ is a $ID_q$-ideal.
\end{lemma}
\begin{proof}
Assume that $q$ is a $n$-th primitive root of unity. From
$\delta_R^{(1)}=\frac{\sigma_q -id}{(q-1)t}$, we get
$$  \sigma_q(a)=(q-1)t (\delta_R^{(1)}(a) -a), \ \mbox{for all} \  a \in I.$$
This shows that $\sigma_q(a) \in I$ for all $a \in I$. Thus $I$ is a
$\sigma_q$-ideal. Conversely, if $I$ is a $\sigma_q$-ideal then it
is a $\delta_R^{(1)}$-ideal. Now, let us consider $a \in \sqrt{I}$.
There exists $m \in \mathbb{N}$ such that $a^m \in I$. But,
$\sigma_q(a^m)=(\sigma_q(a))^m \in I$. Thus $\sigma_q(a) \in
\sqrt{I}$.\\
Now, we will prove by induction that for all $i<n$,
$\delta_R^{(i)}$ stabilizes $\sqrt{I}$.\\
It is true for $i=1$. If it is true for $k<n-1$, then $k<n$ and we
have :$$\delta_R^{(1)} \circ \delta_R^{(k-1)}=\binom{k}{1}_q
\delta_R^{(k)}$$ where $\binom{k}{1}_q \neq 0$ because $k<n$. We
have that $\delta_R^{(1)}$ and $\delta_R^{(k-1)}$ stabilize
$\sqrt{I}$ (by first step and by inductive assumption). Thus
$\delta_R^{(k)}$ stabilizes $\sqrt{I}$. This concludes the proof by
induction.\\
It remains to consider the case where $k=n$. Let $a \in \sqrt{I}$
and $m \in \mathbb{N}$ such that $a^m \in I$. We have:
\begin{equation} \label{eqn:recn} \delta_R^{(nm)}(a^m)= \sum_{i_1+...+i_m=nm}
\sigma_q^{i_2+...+i_m}(\delta_R^{(i_1)}(a))...\sigma_q^{i_m}(\delta_R^{(i_{m-1})}(a))\delta_R^{(i_m)}(a)
.\end{equation} Because $\sigma_q^n=id$, we can rewrite the equation
(\ref{eqn:recn}) as follows $ \delta_R^{(nm)}(a^m)=
(\delta_R^{(n)}(a))^m + B$ \\
 with $$B= \sum_{i_1+...+i_m=nm}^*
\sigma_q^{i_2+...+i_m}(\delta_R^{(i_1)}(a))...\sigma_q^{i_m}(\delta_R^{(i_{m-1})}(a))\delta_R^{(i_m)}(a)$$
where $\sum_{i_1+...+i_m=nm}^*$ means that we only consider  the
$(i_1,...,i_m)$ such that there exists at least one $j$ with $i_j
<n$. We have already proved by induction that $\sqrt{I}$ is stable
by $\sigma_q$ and by $\delta_R^{(i)}$ for $i<n$. This implies that
$B \in \sqrt{I}$. Then $(\delta_R^{(n)}(a))^m $ belongs to
$\sqrt{I}$ since $ \delta_R^{(nm)}(a^m)\in I$ because $I$ itself  is
an $ID_q$-ideal. It follows
$\delta_R^{(n)}(a) \in \sqrt{I}$.\\

So we have proved that $\sqrt{I}$ is stable under $\delta_R^{(k)}$
for all $k \leq n$.  Using the formula $\delta_R^{(i)} \circ
\delta_R^{(k-i)}=\binom{k}{i}_q \delta_R^{(k)}$ and an inductive
proof, we easily show that $\sqrt{I}$ is stable under
$\delta_R^{(k)}$ for all $k \notin n \mathbb{N}$. The proof for  $k
\in n\mathbb{N}$ is an analogue of the case $k=n$. Therefore
$\sqrt{I}$ is an $ID_q$-ideal. \end{proof}

\begin{remark}[Classical case]
For $q$ not equal to a root of unity, the proof of the previous
lemma is more elementary (see Lemma 1.7 in \cite{VPS2}). The reason
is that if $I$ is a $\sigma_q$-ideal then its radical is obviously a
$\sigma_q$-ideal because $\sigma_q$ is an automorphism.
\end{remark}

\subsubsection{Extending iterative $q$-difference operator}

\begin{prop}\label{prop:ext}
Let $R$ be an integral domain, and let $S \subset R$ be a
multiplicatively closed subset of $R$ \textbf{stable under the
action of }$\sigma_q$ such that $0 \notin S$. Let $\delta_R^*$ be an
iterative $q$-difference operator  on $R$. Then there exists a
unique iterative $q$-difference operator $\delta_{S^{-1}R}^*$
extending $\delta_R^*$ to $S^{-1}R$.
\end{prop}
\begin{proof}
Because $\delta_R^*$ is an iterative $q$-difference operator, the
application  $\bold{T}:R \mapsto (R^{\sigma_q}[[T]],+,*)$ defined by
$a \mapsto \bold{T}_a(T)$ is a ring homomorphism (see
\ref{prop:tay}). Since $R$ is commutative, we  have
$$\bold{T}_{ab}(T)= \bold{T}_a(T)* \bold{T}_b(T)=
\bold{T}_b(T)* \bold{T}_a(T) \ \mbox{for all} \ a,b \in R.$$ This
allows us to define the quotient
${\frac{\bold{T}_a(T)}{\bold{T}_b(T)}}^*$ of $\bold{T}_a(T)$ by
$\bold{T}_b(T)$ with respect to the multiplication $*$ for all
$(a,b) \in R \times R^*$. Thereby, the map $\bold{T}$  uniquely
extends to a homomorphism $\tilde{\bold{T}}: S^{-1}R \mapsto
((S^{-1}R)^{\sigma_q}[[T]],+,*)$ via $\frac{a}{b} \mapsto
\tilde{\bold{T}}_{\frac{a}{b}}(T):=
{\frac{\bold{T}_a(T)}{\bold{T}_b(T)}}^*$.   Define
$\delta_{S^{-1}R}^{(k)}(\frac{a}{b})$ to be the coefficient of $T^k$
in $\tilde{\bold{T}}_{\frac{a}{b}}(T)$. Then the collection of maps
$(\delta_{S^{-1}R}^{(k)})_{k \in \mathbb{N}}$ of $S^{-1}R$ to itself
 satisfy  conditions  $1$ and $2$ of Proposition
\ref{prop:tay}. Thus $(\delta_{S^{-1}R}^{(k)})_{k \in \mathbb{N}}$
is an iterative $q$-difference operator for $S^{-1}R$. We also have
$$\tilde{\bold{T}}_{\delta_{S^{-1}R}^{(k)}(a)}(T)=\delta_T^{(k)}(\tilde{\bold{T}}_a(T)) \ \mbox{for all} \  a \in R, \ k \in
\mathbb{N}.$$ The Taylor series  associated to both sides of the
previous equation extend uniquely to $(S^{-1}R)^{\sigma_q}[[T]]$ and
since they coincide on $R^{\sigma_q}[[T]]$, they have to be equal.
Then
$$\tilde{\bold{T}}_{\delta_{S^{-1}R}^{(k)}(a)}(T)=\delta_T^{(k)}(\tilde{\bold{T}}_a(T)) \ \mbox{for all} \  a \in S^{-1}R, \ k \in
\mathbb{N}.$$ By Proposition \ref{prop:tay}, we get that
$(\delta_{S^{-1}R}^{(k)})_{k \in \mathbb{N}}$ is an iterative
$q$-difference operator of $S^{-1}R$ which uniquely extends
$(\delta_R^{(k)})_{k \in \mathbb{N}}$. \end{proof}

\begin{remark}
Let $(R,\delta_R^*)$ be an integral iterative $q$-difference ring.
It is obvious that the set $S$ of  non zero divisors of $R$ is a
multiplicatively closed set and moreover stable under the action of
$\sigma_q$.
\end{remark}
\begin{remark}

In this paragraph we did not mention the possibilities of extending
an iterative $q$-difference operator over a field $K$ to a finitely
generated  separable field extension $E/K$. In fact, this problem
appears already in the classical $q$-difference Galois theory :
extending $\sigma_q$ to an algebraic extension gives rise to
uniqueness problems. Here is an example. Consider a difference field
$(K,\sigma_q)$, where $\sigma_q$ is the identity on some
algebraically closed field $C$ containing $\mathbb{Q}$,  $K$
contains a solution $y$ of $\sigma(x)=cx$, where $c \in C$ is
non-zero and is not a root of unity. Moreover assume that $K$ does
not contain the $n$-th roots of $y$ for some $n>1$. Consider the
extension of $K$ given by $b^n=y$. Then $\sigma(b)=rb$, where
$r^n=c$. The possible choices for $\sigma$ on $K(b)$ depend on the
choices of $r$, and there are $n$ possibilities, which give rise to
$n$ non-isomorphic difference field extensions of $K$.\\
But by chance, we will not have to  handle  such kind of extension
till the end of the paper.
\end{remark}

\subsection{The Wronskian determinant}
In  classical Galois theory of $q$-difference equations, there
exists an analogue of the Wronskian called the $q$-Wronskian or the
Casoratian. If we consider a $\sigma_q$-module $\mathcal{M}$ over a
field $K$ and a family $\mathcal{F} :=\lbrace y_1,...,y_m \rbrace$
of elements of $\mathcal{M}$, we will define the $q$-Wronskian of
the family $\mathcal{F}$ as
$$W_q(y_1,...,y_m):=det ( (\sigma_q^{i-1}(y_j))_{1 \leq i,j \leq
m}).$$ The nullity of the $q$-Wronskian gives a criterion for linear
independence of the $y_i$'s (see for instance \cite{Lu} $1.2$). But when
$q$  is a root of unity, the $q$-Wronskian could vanish for other
 reasons (for instance because $ \sigma_q^n =id$). Thus, we have to change the notion of
 $q$-Wronskian for iterative $q$-difference operators in order to get a
 similar criterion to the one in the classical theory.

\begin{theorem}\label{theorem:indT}
Let $(K,\delta_K^*)$ be an iterative $q$-difference field with field
of constants $C$. Then for any elements  $x_1,...,x_r$ of $K$
linearly independent over $C$, the iterative Taylor series
$\bold{T}_{x_1},...,\bold{T}_{x_r}$ are linearly independent over
$K$.
\end{theorem}
\begin{proof}
 This statement is obviously true for $r=1$. We will
proceed by induction on $r$. Let $(H_r)$ be the hypothesis of
induction, i.e., for any elements  $x_1,...,x_r$ of $K$ linearly
independent over $C$, the iterative Taylor series
$\bold{T}_{x_1},...,\bold{T}_{x_n}$ are linearly independent over
$K$. Suppose that $(H_{r-1})$ is true and let  $x_1,...,x_r \in K$
be  linearly independent over $C$. Assume that
$\bold{T}_{x_1},...,\bold{T}_{x_r}$ are linearly dependent over $K$,
i.e. :
$$ \bold{T}_{x_r} =\sum_{j=1}^{r-1} a_j \bold{T}_{x_j}$$ where $a_j
\in K$ not all equal to zero. This relation implies  that
\begin{equation}\label{eqn:r}\delta^{(k)}(x_r)=\sum_{j=1}^{r-1} a_j \delta^{(k)}(x_j) \ \mbox{for all} \ k \in \mathbb{N} \end{equation}

We will prove that $ \sigma_q(a_j)=a_j$ for all $1 \leq j \leq r-1$.
First of all, let us remark that if $x_1,...,x_{r-1} \in K$ are
linearly independent over $C$ then
$\sigma_q(x_1),...,\sigma_q(x_{r-1}) \in
K$ are linearly independent over $C$.\\
Because of $\delta^{(1)}=\frac{\sigma_q -id}{(q-1)t}$ and from
Equation (\ref{eqn:r}), we have :
$$ \sigma_q(\delta^{(k)}(x_r)) -\delta^{(k)}(x_r)=\sum_{j=1}^{r-1} a_j
\sigma_q(\delta^{(k)}(x_j)) - \sum_{j=1}^{r-1} a_j
\delta^{(k)}(x_j)$$ and
$$\sigma_q(\delta^{(k)}(x_r))=\sum_{j=1}^{r-1} \sigma_q(a_j)
\sigma_q(\delta^{(k)}(x_j)).$$ We also obtain that
$$ \sum_{j=1}^{r-1} (\sigma_q(a_j)-a_j)
\sigma_q(\delta^{(k)}(x_j))=0$$ for all $k\in \mathbb{N}$. Because
$\sigma_q(\delta^{(k)}(x_j))=\frac{1}{q^k}\delta^{(k)}(\sigma_q(x_j))$,
we get $$ \sum_{j=1}^{r-1} (\sigma_q(a_j)-a_j)
(\delta^{(k)}(\sigma_q(x_j)))=0$$ for all $k\in \mathbb{N}$. This
means that $\sum_{j=1}^{r-1}
(\sigma_q(a_j)-a_j)\bold{T}_{\sigma_q(x_j)}=0$. Since
$x_1,...,x_{r-1} \in K$ are linearly independent over $C$,
$\sigma_q(x_1),...,\sigma_q(x_{r-1}) \in K$ are linearly independent
over $C$. Thus we can apply the induction hypothesis $(H_{r-1})$ to
the set of elements $\sigma_q(x_1),...,\sigma_q(x_{r-1})$ of $K$ and
so $\sigma_q(a_j)=a_j$ for $ 1 \leq j \leq r-1$ as desired. \\
For all $k,i \in \mathbb{N}$, we have
$$\binom{i+k}{k}_q
\delta^{(i+k)}(x_r)=\delta^{(i)}\delta^{(k)}(x_r)=\sum_{j=1}^{r-1}\sum_{l=0}^i
\sigma_q^{i-l}(\delta^{(l)}(a_j))\binom{i+k-l}{k}_q\delta^{(i+k-l)}(x_j)$$
and $$\binom{i+k}{k}_q \delta^{(i+k)}(x_r)=\binom{i+k}{k}_q
\sum_{j=1}^{r-1} a_j \delta^{(k)}(x_j).$$ Because
$\sigma_q(a_j)=a_j$ for $ 1 \leq j \leq r-1$, the term for $l=0$ on
the right hand side is equal to the left hand side, thus
\begin{equation}\label{eqn:rec} \sum_{j=1}^{r-1}\sum_{l=1}^i
\sigma_q^{i-l}(\delta^{(l)}(a_j))\binom{i+k-l}{k}_q\delta^{(i+k-l)}(x_j)=0.
\end{equation}
For $i=1$, we deduce from equation (\ref{eqn:rec}) that
$$\sum_{j}^{r-1}\delta^{(1)}(a_j)\delta^{(k)}(x_j)=0.$$
 By applying $\delta^{(1)}$, we obtain :
$$\sum_{j}^{r-1}\sigma_q(\delta^{(1)}(a_j))\delta^{(1)}(\delta^{(k)}(x_j)) + \sum_{j}^{r-1}\delta^{(1)}(\delta^{(1)}(a_j))\delta^{(k)}(x_j)=0,$$
i.e., since $\sigma_q^r \delta^{(s)}= \frac{1}{q^{rs}}\delta_R^{(s)}
\sigma_q^r$ for all $r,s \in \mathbb{N}$, and the $a_j$'s are fixed
by $\sigma_q$,
$$\sum_{j}^{r-1}\frac{q(q^{k+1}-1)}{q-1}\delta^{(1)}(a_j)\delta^{(k+1)}(x_j) + \sum_{j}^{r-1}(q+1)(\delta^{(2)}(a_j))\delta^{(k)}(x_j)=0.$$
For $i=2$, we deduce from equation (\ref{eqn:rec}) that
$$\sum_{j}^{r-1} \sigma_q(\delta^{(1)}(a_j))\binom{k+l}{k}_q\delta^{(k+1) }(x_j) +\sum_{j}^{r-1}
\delta^{(2)}(a_j)\delta^{(k)}(x_j)=0.$$ By subtracting this from the
equality above, we find :
$$\sum_{j}^{r-1}
\delta^{(2)}(a_j)\delta^{(k)}(x_j)=0.$$ By induction, the same
arguments yields

$$\sum_{j}^{r-1}
\delta^{(i)}(a_j)\delta^{(k)}(x_j)=0 \ \mbox{ for} \ k \geq 0 \
\mbox{and} \  i \geq 1.$$
 This leads to

$$\sum_{j}^{r-1}
\delta^{(i)}(a_j)\bold{T}_{x_j}=0.$$ By  hypothesis of induction
$(H_{r-1})$, this implies that $\delta^{(i)}(a_j)=0$ for all $i \geq
1$ and all $1 \leq j \leq r-1$. Hence all the $a_j$'s are constants
and lie in $C$. But we have $x_n = \sum_{j=1}^{r-1}a_j x_j$ (see
Equation (\ref{eqn:rec}) for $k=0$) and thus by assumption of
$C$-linearly independence of $x_1,...,x_r$, we get that $a_j=0$ for
all $1 \leq j \leq r-1$. \end{proof}

\begin{coro}\label{coro:indt}
Let $x_1,...,x_r \in K$ linearly independent over $C$, there exist numbers
$d_1,...,d_r \in \mathbb{N}$ such that
$$det ((\delta^{(d_i)}(x_j))_{i,j=1,...,r})\neq 0.$$
\end{coro}

\begin{defin}
Let $(K,\delta_K^*)$ be an $ID_q$ field with $C(K)=C$ and let
$x_1,...,x_r \in K$ be  linearly independent over $C$. The smallest
numbers $d_1,...,d_r \in \mathbb{N}$ (in lexicographical order) such
that $det ((\delta^{(d_i)}(x_j))_{i,j=1}^r)\neq 0$ (which exist by
Corollary \ref{coro:indt}) are called the  \textbf{difference
orders} of $x_1,...,x_r$. The determinant
$$ wr(x_1,...,x_r):=det ((\delta^{(d_i)}(x_j))_{i,j=1}^r)$$
is called the \textbf{Wronskian determinant} of $x_1,...,x_r$.
\end{defin}

\section{Iterative $q$-difference modules and Equations}
Until the end of this article, we will  assume that $q$ is a $n$-th
primitive root of unity contained in an algebraically closed field
$C$. But we do not
make any assumption about the characteristic of the field $C$.\\
In Section $2$, we have defined iterative $q$-difference rings.
Following the classical way, we extend this concept to modules, in
order to get a suitable notion of iterative $q$-difference equations
associated to these modules.

\begin{defin}\label{defin:mod}
Let $(R,\delta_R^*)$ be an iterative $q$-difference ring.
 Let $M$ be a free $R$-module of finite type over $R$. We will say that $(M, \delta_M^*)$ is an
\textbf{iterative $q$-difference module} if there exists a family of
maps $\delta_M^* =(\delta_M^{(k)})_{k \in \mathbb{N}}$, such that
for all $i,j,k \in \mathbb{N}$
\begin{enumerate}
\item $\delta_M^{(0)}=id_M$,
\item $\phi_M := (q-1)t \delta_M^{(1)} +id_M$ is a bijective map from $M$ to $M$,
\item $\delta_M^{(k)}$ is an additive map from $M$ to $M$,
\item $\delta_M^{(k)}(am) = \sum_{i+j=k} \sigma_q^i(\delta_R^{(j)}(a))
\delta_M^{(i)}(m)$ for $a \in R$ and $m \in M$,
\item $\delta_M^{(i)} \circ \delta_M^{(j)}=\binom{i+j}{i}_q
\delta_M^{(i+j)}$.
\end{enumerate}
The set of all iterative $q$-difference modules over $R$ is denoted
by $IDM_q(R)$.

\end{defin}
\begin{remark}[Classical case]\label{remark:qdiff}

 If $q$ is not a root of unity, it is easy to see that
$\phi_M(am)= \sigma_q(a) \phi_M(m)$ for all $a \in R$ and $m \in M$.
Moreover, $\delta_M^{(k)}=\frac{{\delta_M^{(1)}}^k}{[k]_q!}$. Thus,
in the case where $q$ is not a root of unity,  an $ID_q$-module is
nothing else than a $q$-difference module in the sense of
\cite{VPS2} $1.4$.
\end{remark}
As in \ref{remark:ech}, we easily show that we have for all $j,i \in
\mathbb{N}$,
\begin{equation}\label{eqn:sd1}\phi_M^j \delta_M^{(i)}=
\frac{1}{q^{ji}}\delta_M^{(i)} \phi_M^j .\end{equation}

\begin{defin}
Let $(M, \delta_M^*)$ and $(N, \delta_N^*)$ be two iterative
$q$-difference modules over $R$ and let $\phi \in Hom_R(M,N)$. We
will say that $\phi$ is an \textbf{iterative $q$-difference
homomorphism} if $ \delta_N^{(k)}\circ \phi=\phi \circ
\delta_M^{(k)}$ for all $k \in \mathbb{N}$.
\end{defin}
\begin{defin}
Let $(R,\delta_R^*)$ be an iterative $q$-difference ring.
 Let $(M, \delta_M^*)$ be an
iterative $q$-difference module over $R$. The $C(R)$-module
$$V_M:= \bigcap_{k \in \mathbb{N}} Ker(\delta_M^{(k)})$$ is called the
\textbf{solution space} of the iterative $q$-difference module $M$.
We will say that   $M$ is a \textbf{trivial} iterative
$q$-difference module if $M \simeq V_M \otimes_{C(R)}R$.
\end{defin}

\begin{theorem}\label{theorem:tannak}
Let $(L,\delta_L^*)$ be an iterative $q$-difference field. Let us
denote by $IDM_q(L)$ the category with objects  the iterative
$q$-difference modules over $L$ and  morphisms  the iterative
$q$-difference morphisms. Then $IDM_q(L)$ is a neutral Tannakian
category over $C(L)$. The unit object is $(L, \delta_L^{*})$.
\end{theorem}
\begin{proof}
We refer to \cite{M} Theorem 2.5. for the fact that $IDM_q(L)$ is an abelian category, the case for iterative differential modules
being the same as the one of iterative $q$-difference modules. For
$M$ and $N$ two objects of  $IDM_q(L)$, we define the tensor product
$M\otimes N:= M \otimes_L N$ by the usual tensor product as
$L$-modules and turn it to an $ID_q$-module via

$$\delta_{M \otimes N}^{(k)}(x \otimes y)= \sum_{i+j=k}
\phi_M^{j}(\delta_M^{(i)}(x))\otimes \delta_N^{(j)}(y)$$ for all $x
\in M, y \in N$. The proof that $(\delta_{M\otimes N}^{(k)})_{k \in
\mathbb{N}}$ is an iterative $q$-difference operator on $M \otimes
N$ is   analogous to
the proof of  Proposition \ref{prop:tens}.\\
The dual of an object $M$ of $IDM_q(L)$ is then given by
$M^*=Hom_L(M,L)$ together with
$$
\delta_{M^*}^{(k)}(f)=\sum_{i+j=k}(-1)^i
q^{\frac{i(i+1)}{2}}\sigma_q^i(\delta_L^{(j)})\circ f \circ
\delta_M^{(i)} \circ \phi_M^{-i}$$ for all $f \in M^*$. The proof
that $(M, \delta_{M^*}^{*})$ is
an iterative $q$-difference module is left to the reader. We just recall that if $(M, \phi_M)$ is a $q$-difference module in the sense of \cite{VPS2}, then
 $M^*$ is endowed with a $q$-difference module structure via
 $$ \phi_{M^*}(f) := \sigma_q \circ f \circ \phi_M^{-1}.$$\\
The evaluation map $\epsilon: M\otimes M^* \rightarrow
\bold{1}_{IDM_q(L)}=L$ sends $x\otimes f$ to $f(x)$, and the
coevaluation map $\eta: L \rightarrow M^* \otimes M$ is defined by
mapping $1$ to $\sum_{i=1}^n x_i^* \otimes x_i$, where $\lbrace x_i
\rbrace _{i=1}^n$ denotes an $L$-basis of $M$ and $\lbrace x_i^*
\rbrace _{i=1}^n$ the associated dual basis of $M^*$. Note that the
definition of $\eta$ does not depend on the  chosen basis. It
remains to show that $\epsilon$ and $\eta$ are $ID_q$-homomorphism
and that they satisfy $(\epsilon \otimes id_M) \circ (id_M \otimes
\eta)=id_M$ and $(id_{M^*}\otimes \epsilon) \circ (\eta \otimes
id_{M^*})=id_{M^*}$ for all objects $M$ of $IDM_q(L)$. We have
$$\begin{array}{lll} \epsilon \circ \delta_{M \otimes M^*}^{(k)}(x
\otimes f)& =&\epsilon\left(\sum_{i+j=k} \delta_M^{(i)}(x) \otimes
\phi_{M*}^i(\delta_{M^*}^{(j)}(f))\right)=\sum_{i+j=k}\phi_{M^*}^i(\delta_{M^*}^{(j)}(f))(\delta_M^{(i)}(x))\\
& = &\sum_{i+j=k}\sum_{l=0}^j(-1)^l
q^{l(l+1)/2}\sigma_q^{i+l}(\delta_L^{(j-l)})\circ f \circ
\delta_M^{(l)}\circ \phi_M^{-(i+l)}(\delta_M^{(i)}(x))\\
& = & \sum_{i+j=k}\sum_{l=0}^j(-1)^l q^{l(l+1)/2}
\sigma_q^{l+i}(\delta_L^{(j-l)})\circ f \circ
q^{i(i+l)}\binom{i+l}{i}_q \delta_M^{(i+l)}(
\phi_M^{-(i+l)}(x))\end{array}$$ and thus
$$\epsilon \circ \delta_{M \otimes M^*}^{(k)}(x \otimes
f)=\sum_{i_*+j_* = k} \sigma_q^{i_*}(\delta_L^{(j_*)})\circ f \circ
\delta_M^{(i_*)}( \phi_M^{-i_*}(x)) \left(\sum_{i=0}^{i_*} (-1)^i
q^{i(i-1)/2} \binom{i_*}{i}_q \right).$$ By expanding $(1;q)_{i_*}$,
we see that the inner sum equals zero if and only if $i_* \neq 0$.
We thus get
$$\epsilon \circ \delta_{M \otimes M^*}^{(k)}(x \otimes
f)=\delta_L^{(k)} (f (x)) = \delta_L^{(k)} \circ \epsilon(x \otimes
f).$$ The proof for $\eta$ is analogous.

Let $x=\sum_{i=1}^n a_ix_i  \in M$, then $(\epsilon \otimes id_M)
\circ (id_M \otimes \eta)(x)=\epsilon \otimes id_M(x \otimes
(\sum_{i=1}^n x_i^* \otimes x_i))=\epsilon \otimes id_M(
\sum_{i=1}^n(x \otimes x_i^*)\otimes x_i)=\sum_{i=1}^n
x_i^*(x)\otimes x_i =\sum_{i=1}^n a_ix_i=x.$ Again, the second
statement is proved analogously. Finally, we note that

$$End_{IDM_q(L)}(\bold{1}_{IDM_q(L)})=End_{ID_q}(L)=C(L),$$
finishing the proof.\end{proof}

\textbf{Link with the iterative differential modules}\\

In this paragraph, we will show that iterative $q$-difference
operators  and iterative derivations are closely related.
\begin{prop}\label{prop:compder}
Let $q$ be a primitive $n$-th root of unity. Let $(L,\delta_L^*)$ be an
iterative $q$-difference field and let $(M,\delta_M^*)$ be an
iterative $q$-difference module over $L$.
 Set $L_0 = \cap_{j \notin n \mathbb{N}}
Ker(\delta_L^{(j)})$ and $M_0 = \cap_{j \notin n \mathbb{N}}
Ker(\delta_M^{(j)})$. Then $(M_0, (\partial_M^{(k)}:=\delta_M^{(nk)})_{k \in
\mathbb{N}})$ is an iterative differential module over the iterative differential field
$(L_0, (\partial^{(k)}:= \delta_L^{(nk)})_{k \in \mathbb{N}})$)
 (see
\cite{M} Definition $2.1$).
\end{prop}
\begin{proof}[Hint of proof] For instance, we will prove point $2$ of definition $2.1$ in \cite{M}, that is
$$\partial_M^{(k)}(am)=\sum_{i+j=k}\partial^{(i)}(a)\partial_M^{(j)}(m) \ \mbox{with} \ (a,m) \in L_0 \times M_0.$$
We have
$$\partial_M^{(k)}(am):=\delta_M^{(nk)}(am)=\sum_{i+j=nk}\sigma_q^j(\delta_L^{(i)}
(a)) \delta_M^{(j)}(m).$$
Because $(a,m) \in L_0 \times M_0$ we have that $\delta_L^{(i)}
(a)=0= \delta_M^{(j)}(m)$ for all $ i \notin n\mathbb{N} \ \mbox{and} \ j \notin n\mathbb{N}$. then,
$$\partial_M^{(k)}(am):=\delta_M^{(nk)}(am)=\sum_{ni+nj=nk}\sigma_q^{nj}(\delta_L^{(ni)}
(a)) \delta_M^{(nj)}(m)=\sum_{i+j=k}\partial^{(i)}(a)\partial_M^{(j)}(m).$$ The last inequality comes from the
fact that $\sigma_q^n=id$. To prove point $3$ of definition $2.1$ in \cite{M}, we use the same facts that in point $2$ and the formula
$\binom{in}{jn}_q=\binom{i}{j}$.
\end{proof}
Therefore, one could hope, as in \cite{M} Theorem $2.8$ (or in \cite{MP} section $5$), to construct projective
systems deeply related to our iterative $q$-difference module in order to obtain a suitable notion
 of iterative $q$-difference equations. These projective systems could be perhaps seen as some kind of jet spaces
 for the iterative $q$-difference operator.\\
  But
 our situation is slightly different as the one considered in \cite{MP} because we treat simultaneously fields of positive and zero characteristic.\\

  In the case of characteristic zero, we
may regain all the iterative $q$-difference operators only with the
knowledge of $\delta_M^{(1)}$ and $\delta_M^{(n)}$. This is due to
the formula $( \delta_M^{(n)})^{n^{k-1}}=(n^{k-1})!
\delta_M^{(n^k)}$ and to the fact that  the family $\lbrace
\delta_M^{(1)},(\delta_M^{(n^k)})_{k \in \mathbb{N}}\rbrace$
generates the iterative $q$-difference operator.  Therefore
 we will only obtain degenerated projective systems but  this
 is not  a hindrance to the construction of iterative
$q$-difference equations in characteristic $0$ (see section \ref{subsec:car0}).\\
In positive
characteristic,  the whole family $\lbrace
\delta_M^{(1)}, (\delta_M^{(np^k)})_{k \in \mathbb{N}}\rbrace$ and not less is necessary to
recover the iterative $q$-difference operator. In this situation, we will show  that  the category
of iterative $q$-difference modules is
equivalent to the category of some specific projective systems (see section \ref{subsec:carp}).
 This is a very nice tool because it allows us to translate
our computations from the non commutative world of iterative
$q$-difference modules to the world of linear algebra, via the
vector spaces associated to the projective systems.\\
This comparison between iterative differential modules and specific
projective systems already  appears  in the work of B.H. Matzat and M.
van der Put. But  to obtain an equivalence of category between the one of
projective systems linked to iterative derivations and the one
associated to iterative $q$-difference, we need to have $q^p=1$ and
this assumption makes no sense. A hope for realizing this
equivalence will be perhaps to rebuild both theories over
non-algebraically closed base rings, such as
$\mathbb{Z}/p^m\mathbb{Z}$ and try to reach
the Witt vectors. But this is a future research topic.\\

\subsection{Case of characteristic $p$}\label{subsec:carp}
\subsubsection{Projective systems}
~~\\
 Let $(L,\delta_L^*)$ be an
iterative $q$-difference field  of characteristic $p$ and let
$(M,\delta_M^*)$ be an iterative $q$-difference module over $L$. In
positive characteristic, we have the exact analogue of the
equivalence of categories obtained by Matzat in
\cite{M} Theorem $2.8$.\\
 Put  $L_{k+1} = \cap_{0
\leq j <k}
Ker(\delta_L^{(np^j)})\cap Ker(\delta_L^{(1)})$ for $k >0 $ and $L_0=L$.\\
Put   $M_{k+1} = \cap_{0 \leq j <k}
Ker(\delta_M^{(np^j)}) \cap Ker(\delta_M^{(1)})$  for all  $k >0$ and $M_0=M$.\\
\begin{prop}\label{prop:projp}
We have,
\begin{enumerate}
\item $M_k$ is an $L_k$-vector space of finite dimension.
\item  The inclusion $\phi_k : M_{k+1} \hookrightarrow M_k$ is $L_{k+1}$-linear and defines
a projective system $(M_k,\phi_k)_{k \in \mathbb{N}}$.
\item The map $\phi_k$ extends to an isomorphism of $L_k$-vector-spaces from $M_{k+1} \otimes L_k$ to $M_k$. \end{enumerate}
\end{prop}

\begin{proof}
The two first statements are obvious. Let us prove the third one.
For all $k \in \mathbb{N}$, $L_k \otimes_{L_{k+1}} M_{k+1} \subset M_k$, thus  $dim_{L_{k+1}}(M_{k+1}) \leq
dim_{L_{k}}(M_{k})$. On the other hand, $M_k$ is an $L_k$-vector space and hence an $L_{k+1}$-vector  space since $L_{k+1} \subset L_k$.
For  $k \geq 1$, the application
$\delta_M^{(np^{k-1})}$ is $L_{k+1}$-linear on  $M_{k}$  and
$(\delta_M^{(np^{k-1})})^p=0$, so
$dim_{L_{k+1}}(M_{k+1})=dim_{L_{k+1}}(\rm{Ker}(\delta_M^{(np^{k-1})})|_{M_{k}})\geq
\frac{1}{p}dim_{L_{k+1}}(M_{k})\geq dim_{L_{k}}(M_{k})$, where
the last inequality comes
 from the fact that $\delta_L^{(np^{k-1})}$ is an $L_{k+1}$-linear endomorphism
 of $L_{k}$ of order of nilpotence $p$. \\
 For $k=0$,
we   have $(\delta_M^{(1)})^n=0$.
 Therefore,
 $dim_{L_1}(M_1)=dim_{L_1}(Ker(\delta_M^{(1)}|_{M})\geq
 \frac{1}{n}dim_{L_1}(M) \geq dim_L(M)$, where  the last inequality comes
 from the fact that $\delta_L^{(1)}$ is an $L_1$-linear endomorphism
 of $L$ of order of nilpotence $n$ ($q$ is a $n$-th primitive root of unity).\end{proof}
\subsubsection{Equivalence of categories}
\begin{Not}\label{Not:projp}
 Let $(L,\delta_L^*)$ be an iterative $q$-difference field of characteristic $p$. Let us denote by $Proj_q(L)$ the category of projective systems $(N_k,\psi_k)_{k \in
\mathbb{N}}$ over $L$ with the properties:
\begin{enumerate}
\item $N_k$ is an $L_k$-vector space of finite dimension and
$\psi_k$ is $L_{k+1}$-linear,
\item each $\psi_k$ uniquely extends to an $L_k$-isomorphism
$$ \xymatrix{ \tilde{\psi}_k: L_k \otimes_{L_{k+1}} N_{k+1} \ar[r] & N_k.}$$
\end{enumerate}
\end{Not}

\begin{theorem}\label{theorem:equip}
Let $(L,\delta_L^*)$ be an iterative $q$-difference field of
 positive characteristic. Then the category $Proj_q(L)$ is equivalent to
the category $IDM_q(L)$.

\end{theorem}
\begin{proof}
We already saw in Proposition \ref{prop:projp} how an object of
$IDM_q(L)$ leads to  an object of $Proj_q(L)$. Conversely, let us
consider $(N_k,\psi_k)_{k \in \mathbb{N}}$ in the category
$Proj_q(L)$.  We will now construct its associated iterative
$q$-difference
module.\\
Put $M_0:=N_0$ and define $M_k:=\psi_0 \circ \psi_1\circ
...\psi_{k-1}(N_k)$. Then $M_{k+1} \subset M_{k} \subset ...\subset
M_0$. Let $B_k=\lbrace b_1,...,b_m \rbrace$ be an $L_k$-basis for
$M_k$, then by property $2$ of Notations \ref{Not:projp}, $B_k$ is an $L$-basis of $M=M_0$. Let $x
\in M$, there exits $(\lambda_i)_{i=1,...,m} \in L^m$ such that
$x=\sum_{i=1}^m \lambda_i b_i$. Then, for all $ j < n p^{k-1}$, set
$$\delta_M^{(j)}(x):=\sum_{i=1}^m \delta_L^{(j)}(\lambda_i) b_i.$$
This is possible because we want  $B_k$ to lie in the kernel of
$\delta_M^{(j)}$ for
$j<np^{k-1}$. Because all change of basis are with coefficients in $L_k$, this definition is independent
of the choice of the $L_k$-basis of $M_k$. Therefore,   $(M_0,\delta_{M_0}^*)$ is an object $IDM_q(L)$.\\
Let us consider two objects $\mathcal{M}:=(M_k,\phi_k)_{k \in
\mathbb{N}}$ and $ \mathcal{N}:=(N_k,\psi_k)_{k \in \mathbb{N}}$ of
$Proj_q(L)$ and $\alpha$ a morphism from $\mathcal{M}$ to
$\mathcal{N}$ in the category $Proj_q(L)$, i.e. $\alpha_k$ is $L_k$
linear and the diagram
$$\xymatrix{
M_k \ar[r]^{\alpha_k} & N_k \\
M_{k+1} \ar[u]^{\phi_k} \ar[r]^{\alpha_{k+1}} & N_{k+1}
\ar[u]^{\psi_k} }$$ is commutative. Then we have $\delta_N^* \circ
\alpha_0= \alpha_0 \circ \delta_M^*$. Also, with this property, it
is then easy to verify that
$$\xymatrix{
Proj_q(L) \ar[r]  & IDM_q(L) \\
(M_k,\phi_k) \ar@{|->}[r] & (M_0,\delta_{M_0}^*)}$$ (with
$\delta_{M_0}^*$ as defined  above) is in fact an equivalence of
categories.
\end{proof}
\subsubsection{Iterative $q$-difference equations in positive characteristic}
~~\\
As we expect from standard $q$-difference Galois theory, any
iterative $q$-difference module should give rise to an iterative
$q$-difference equation consisting of a family of equations.
Proposition \ref{prop:qeq}  shows how to
obtain this equation from a given $ID_q$-module over a field of positive characteristic.

\begin{prop}
Let $(L,\delta_L^*)$ be an iterative $q$-difference field of
characteristic $p$ and let $(M, \delta_M^*)$ be an object of
$IDM_q(L)$. Let us consider the  projective system
$(M_k,\phi_k)_{k \in \mathbb{N}}$ associated to $M$ as in $4.1.1$. For all $k \in
\mathbb{N}$, let us denote by $B_k$ a basis of $M_k$ as an  $L_k$-vector space (written as a row) and let
 $D_k \in Gl_n(L_k)$ (with $n=dim_L M$) be the matrix
of $\phi_k$ with respect to that basis, i.e., $ B_kD_k=B_{k+1}.$\\

Then, for any $l\in \mathbb{N}^*$ and for any $X \in L^n$, we have :
\begin{enumerate}
\item $B_0 X= B_l X_l$ where $X_l=D_{l-1}^{-1}...D_0^{-1}X$,
\item $\delta_M^{(k)}(B_0X)=B_l\delta_L^{(k)}(X_l)=B_0D_0...D_{l-1} \delta_L^{(k)}(D_{l-1}^{-1}...D_0^{-1}X)$ for $0<k< n p^{l-1}$.
\end{enumerate}

\end{prop}

\begin{proof}
Part $1$ is obvious by definition. For part $2$ we have
$$\delta_M^{(k)}(B_0X)=\delta_M^{(k)}(B_lX_l)=B_l\delta_L^{(k)}(X_l)=B_0D_0...D_{l-1} \delta_L^{(k)}(D_{l-1}^{-1}...D_0^{-1}X)
\ \mbox{for} \ 0< k<np^{l-1},$$ using the definition of $B_l$
\end{proof}

\begin{prop}\label{prop:qeq}
 Let $\bold{y} \in L^n$ and  $B_0=\lbrace b_1,...,b_n
\rbrace$ be a basis of $M$. The following statements are equivalent
\begin{enumerate}
\item Let $\bold{y} \in L^n$.  $B_0\bold{y}=\sum_{i=1}^n y_i b_i \in V_M=\cap_{k \in
\mathbb{N}} M_k$.
\item For all $l \in \mathbb{N}^*$, we have $\delta_L^{(k)}(\bold{y}_l)=0$ for $0 < k<np^{l-1}$, where
$\bold{y}_l=D_{l-1}^{-1}...D_0^{-1}\bold{y}$.

\item $$\delta_L^{(np^k)}(\bold{y})=\tilde{A}_{k} \bold{y}, \ \mbox{  for all} \  k \geq 0$$ where $\tilde{A}_{k}=
\delta_L^{(np^k)}(D_0...D_{k+1})(D_0...D_{k+1})^{-1}$ and
$\delta_L^{(1)}(\bold{y})=A_1 \bold{y}$ where $A_1=
\delta_L^{(1)}(D_0)(D_0)^{-1}$.
\end{enumerate}
\end{prop}

\begin{proof}
First, we show that statements $1$ and $2$ are equivalent :
$B_0\bold{y} \in V_M$ if and only if $\delta_M^{(k)}(B_0\bold{y})=0$
for all $k \in \mathbb{N}^*$. The claim is obvious by using the
equation
$$\delta_M^{(k)}(B_0\bold{y})=B_l\delta_L^{(k)}(\bold{y}_l)$$ which holds for
$0<k<np^{l-1}$ (see the Proposition $3.10$).\\
Finally, the equivalence of $2$ and $3$ is obtained using:

$$
\delta_L^{(np^l)}(\bold{y})=\delta_L^{(np^l)}(D_0...D_{l+1}
\bold{y}_{l+2})=\delta_L^{(np^l)}(D_0...D_{l+1}) \bold{y}_{l+2} +
D_0...D_{l+1}\delta_L^{(np^l)}(\bold{y}_{l+2})$$ $$=\delta_L^{(np^l)}(D_0...D_{l+1})(D_0...D_{l+1})^{-1} \bold{y}=
\tilde{A}_{l}\bold{y}$$ where $\delta^{(np^l)}(\bold{y}_{l+2})=0$ and
$$
\delta_L^{(1)}(\bold{y})=\delta_L^{(1)}(D_0
\bold{y}_{1})=\delta_L^{(1)}(D_0) \bold{y}_{1} +
\sigma_q(D_0)\delta_L^{(1)}(\bold{y}_{1})=A_1\bold{y}.$$
\end{proof}
\begin{defin}

The family of equations $\lbrace \delta_L^{(1)}(\bold{y})=A_1
\bold{y},    \delta_L^{(np^k)}(\bold{y})=\tilde{A}_k \bold{y} \rbrace_{  k
 \geq 0}$ related to the $IDM_q$-module $(M, \delta_M^*)$ by Proposition
\ref{prop:qeq} is called an \textbf{ iterative $q$-difference
equation} ($ID_qE$).
\end{defin}

We  give below  some examples of iterative $q$-difference equations over
fields of  positive characteristic.
\begin{ex}
Let $p$ be a prime number, let $C=\overline{\mathbb{F}_p}$ be an
algebraic closure of $\mathbb{F}_p$  and let $F=C(t)$ be the
rational function field with coefficients in $C$. Let $(a_l)_{l \geq
0}$ be a set of elements in $C$. Let $M =Fb_1$. Suppose that,
 $D_{l+1}=(t^{a_lnp^l}) \in Gl_1(F_{l+1})$ for $l\in \mathbb{N}$ and
 $D_0=(1)$. We have

$$\tilde{A}_{k}=\delta_L^{(np^k)}(D_0...D_{k+1})(D_0...D_{k+1})^{-1}=\delta_L^{(n^k)}(t^{\sum_{l=0}^k
a_lnp^l})t^{-\sum_{l=0}^k a_lnp^l}=\frac{a_k}{t^{np^k}}$$ because
$\binom{\sum_{j=0}^ka_jnp^j}{np^k}_q=a_k$. Hence $
\delta_M^{(np^k)}(y)= \frac{a_k}{ t^{np^k}}y$  for all $k \in
\mathbb{N}$.
\end{ex}
\begin{ex}
Let $p$ be a prime number, let $C=\overline{\mathbb{F}_p}$ be an
algebraically  closure of $\mathbb{F}_p$  and let $F=C(t)$ be the
rational function field with coefficients in $C$. Let $(a_l)_{l \geq
0}$ be a set of elements in $C$. Let $M =Fb_1 \oplus F b_2$. Suppose
that,

  $$D_{l+1}:=\left( \begin{array}{cc} 1 & a_lt^{np^l} \\
0 &1 \end{array} \right) \mbox{for all} \ l \in \mathbb{N}$$ and
$$D_{0}:=\left( \begin{array}{cc} 1 & 0 \\
0 &1 \end{array} \right) .$$
Using Proposition \ref{prop:qeq} $3)$, we obtain,\\

$$\tilde{A}_{k}=\left( \begin{array}{cc} 0 & a_k \\
0 &0 \end{array} \right)$$ and $A_1=0$.
 So, the associated $ID_qE$ associated
to $M$ is

$$ \delta^{(np^k)}(Y)= \tilde{A}_k Y= \left( \begin{array}{cc}
0 & a_k \\
0 & 0 \end{array} \right) Y \mbox{ for all} \ k \in \mathbb{N}.$$

\end{ex}

\subsection{Case of characteristic $0$}\label{subsec:car0}
\subsubsection{Projective systems}
~~\\
Let $(L,\delta_L^*)$ be an iterative $q$-difference field  of zero
characteristic and let $(M,\delta_M^*)$ be an iterative
$q$-difference module over $L$. Put, for all  $k \in \mathbb{N}^*$,
$L_k = \cap_{0 \leq j <k}
Ker(\delta_L^{(n^j)})$ and $L_0=L$.\\
Put, for all  $k \in \mathbb{N}^*$, $M_k = \cap_{0 \leq j <k}
Ker(\delta_M^{(n^j)})$ and $M_0=M$.\\
\begin{prop}\label{prop:proj}
\begin{enumerate}
\item $M_k$ is a $L_k$-vector space of finite dimension.
\item $M_k=M_2$ for all $k\geq 2$.
\item  Let $\phi_1$ be  the
embedding $M_{2} \hookrightarrow M_1$. Then the map $\phi_1$
extends to a monomorphism  of $L_1$-vector-spaces from $M_{2}
\otimes L_1$ to $M_1$.
\item Let $\phi_0$ be the
embedding $M_{1} \hookrightarrow M_0$. Then the map $\phi_0$
extends to an isomorphism  of $L$-vector-spaces from $M_{1} \otimes
L$ to $M_0$.

\end{enumerate}
\end{prop}

\begin{proof}
The first statement is obvious. Because $
(\delta_M^{(n)})^{n^{k-1}}=(n^{k-1})! \delta_M^{(n^k)}$ for all $k
\geq 1$ (see part $4$ of Proposition \ref{prop:binform}), we have
$M_k=M_1$ for all $k\geq 2$. The third statement is obvious.\\ We
now prove  the fourth statement. We have $M_1\otimes_{L_1}L \subset M_0$ so $dim_{L_{1}}(M_{1})
\leq dim_{L}(M)$.

Conversely, from  $(\delta_M^{(1)})^n=0$ and $(\delta_L^{(1)})^n=0$
follows
 $$dim_{L_1}(M_1)=dim_{L_1}(Ker(\delta_M^{(1)}|_{M}))\geq
 \frac{1}{n}dim_{L_1}(M) \geq dim_L(M).$$
\end{proof}

\subsubsection{Iterative $q$-difference equations in characteristic zero}
~~\\
 The restriction of $\delta_M^{(n)}$ to the $L_1$-vector space $M_1$ behaves like a connection (see Proposition
 \ref{prop:compder}), i.e. to be more precise
 $$\delta_M^{(n)}(\lambda x)=\lambda \delta_M^{(n)}(x) +\delta_L^{(n)}(\lambda) x \ \mbox{for all} \ (\lambda,x) \in L_1
 \times M_1.$$
 This observation allows us to consider the matrix of $\delta_{M}^{(n)}|_{M_1}$ with respect to an $L_1$-basis of $M_1$ and
 thus to set the following notations.

\begin{Not}\label{Not:qit}
Let  $B_1$ (resp. $B_0$) be a
$L_1$-basis of $M_1$ (resp. a $L_0$-basis of $M_0$). Let $n=dim_L M$.\\
\begin{enumerate}
\item Because of
Proposition \ref{prop:proj}, we have $M_1 \otimes L \simeq M$. Now
let us denote by $D_0 \in Gl_n(L)$ the matrix of $\phi_0$ with
respect to the basis $B_{1}$ and $B_0$, i.e., $ B_0D_0=B_{1}$.
\item Let
$C_n$ be  the matrix of $\delta_{M}^{(n)}|{M_1}$ with respect to the basis
$B_1$, i.e. $$\forall X \in L_1^n, \delta_M^{(n)}(B_1 X)=B_1C_nX + B_1 \delta_L^{(n)}X.$$
 \item Set $A_0:=Id_n, A_1:=\delta_L^{(1)}(D_0)D_0^{-1}$ and define inductively $A_k$ for all $0 < k< n-1$ with
 $$A_{k+1}=\frac{(q-1)(\delta_L^{(1)}(A_k) +\sigma_q(A_k)A_1)}{(q^{k+1}-1)}$$ and $$A_n:=-D_0C_nD_0^{-1}
 -\sum_{k=0}^{n-1}D_0\sigma_q^k(\delta_L^{(n-k)}(D_0^{-1}))A_k.$$
\end{enumerate}
\end{Not}

\begin{prop}\label{prop:qeq0}
Using the previous notation, the following statements are
equivalent:
\begin{enumerate}
\item $B_0\bold{y}=\sum_{i=1}^n y_i b_i \in V_M=\cap_{k \in
\mathbb{N}} M_k=M_1 \cap M_2$.

\item $\delta_L^{(1)}(\bold{y})=A_1
\bold{y} \ \mbox{and} \ \delta_L^{(n)}(\bold{y})=A_n
\bold{y} $, with $A_1, A_n $ defined in Notation \ref{Not:qit}.
\end{enumerate}
\end{prop}
\begin{proof}
 $B_0\bold{y}=\sum_{i=1}^n y_i b_i \in V_M$ if and only if  for all $0 < k\leq n$ we
 have $\delta_M^{(k)}(B_0\bold{y})=0$ (remember that it is sufficient in the case of a base field of characteristic zero to
 consider only the iterative $q$-difference of order $1$ and $n$).\\
Let us first consider the case $k=1$. We have as in Proposition \ref{prop:qeq}

\begin{equation}\label{eqn:rg1} \delta_L^{(1)}(\bold{y})= \delta_L^{(1)}(D_0D_0^{-1} \bold{y})
 = \delta_L^{(1)}(D_0) D_0^{-1}\bold{y}=A_1 \bold{y}.\end{equation}
For $0<j<n-1$, we  proceed by induction and  using Equation (\ref{eqn:rg1}) and $\delta_L^{(1)}\circ \delta_L^{(j)}=\frac{q^{j+1}-1}{q-1} \delta_L^{(j+1)}$
 we have
$\delta_L^{(j+1)}(\bold{y})=A_j \bold{y}$  for all $j<n$ where $A_j=\frac{(q-1)(\delta_L^{(1)}(A_j)
+\sigma_q(A_j)A_1)}{(q^{j+1}-1)}.$ \\
By assumption
$$\delta_M^{(n)}(B_0\bold{y})=0=\delta_M^{(n)}(B_1D_0^{-1}\bold{y})=B_1C_nD_0^{-1}\bold{y}
+B_1\delta_L^{(n)}(D_0^{-1}\bold{y}).$$
and thus $$0=D_0 C_nD_0^{-1}\bold{y} +D_0\sum_{j=0}^{n}
\sigma_q^j(\delta_L^{(n-j)}(D_0^{-1}))\delta_L^{(j)}(\bold{y})= D_0C_nD_0^{-1}\bold{y}
 +\sum_{k=0}^{n-1}D_0\sigma_q^k(\delta_L^{(n-k)}(D_0^{-1}))A_k\bold{y} + \delta_L^{(n)}(\bold{y}).$$ This gives
 $\delta_L^{(n)}(\bold{y})= A_n\bold{y}.$

Hence the first statement implies the second. By going through the
computation backwards, we obtain the equivalence between the two
statements.\end{proof}
\begin{defin}

The family of equations $\lbrace \delta_L^{(1)}(\bold{y})=A_1
\bold{y}, \delta_L^{(n)}(\bold{y})=A_n
\bold{y} \rbrace$ related to the $IDM_q$-module
$(M, \delta_M^*)$ by Proposition \ref{prop:qeq0} is called an
\textbf{ iterative $q$-difference equation}($ID_qE$).
\end{defin}

\begin{ex}\label{ex:car0}
Let $L=\mathbb{C}(t)$ and let $q$ be a $n$-th primitive root of
unity. Let $M =Fb_1$ be a rank one $IDM_q(L)$-module and suppose
that $\Phi(b_1)=b_1$. Then an easy computation leads to  $C_j=0$ for
all $1 \leq j < n$ and $A_j=0$ for $1 \leq j <n$. Now, let $a_1$ be
an integer and set $C_n=\frac{a_1}{t^n}$. Then
$A_n=\frac{-a_1}{t^n}$.
\end{ex}

\section{Iterative $q$-difference Picard-Vessiot extensions}
In this section, we develop a Picard-Vessiot theory for iterative
$q$-difference equations. We build the Picard-Vessiot ring inspired
by the usual construction, but we have to adapt our construction to
a infinite set of variables, and thus some modifications are
necessary.
\subsection{Iterative Picard-Vessiot rings}

\begin{Not}
Let $(L,\delta_L^*)$ be an iterative $q$-difference field. If,
\begin{enumerate}
\item the characteristic of the constants field $C$ of $L$ is zero
then let us denote by $(k_C)$ the family $\{1,n \}$,
\item the characteristic of the constants field $C$ of $L$ is
positive equal to $p$ then let us denote by $(k_C)$ the family
$\lbrace 1,(np^k)_{k \in \mathbb{N}} \rbrace$.
\end{enumerate}
\end{Not}
\begin{remark}[Classical case]\label{remark:eqcl}

As mentioned before, when $q$ is not a root of unity, an iterative
$q$-difference module is the same object as  a $q$-difference
module. Moreover, in this case the iterative $q$-difference equation
is just obtained  by considering the equation of level  $1$ and if
there exists $Y \in Gl_n(R)$ such that $\delta_L^{(1)}(Y)=A_1 Y$
then  for all $k \in \mathbb{N}$  we have $\delta_L^{(k_C)}(Y)=A_k
Y$. Thereby, when $q$ is not a root of unity, an iterative
$q$-difference equation is simply a $q$-difference equation in the
sense of \cite{VPS2} p.$5$.
\end{remark}

\begin{defin}\label{defin:pv}
Let $(L,\delta_L^*)$ be an iterative $q$-difference field, let $(M,
\delta_M^*)$ be an object of $IDM_q(L)$, and let $\lbrace
\delta_L^{(k_C)}(\bold{y})=A_k \bold{y} \rbrace_{  k \in
\mathbb{N}}$ be  an \textbf{ iterative $q$-difference equation}
related to the
$IDM_q$-module $(M, \delta_M^*)$, denoted by $ID_qE(M)$.\\
Let $(R,\delta_R^*)$ be an iterative $q$-difference extension of
$(L,\delta_L^*)$. A matrix $Y \in Gl_n(R)$ is called a
\textbf{fundamental solution  matrix} for $ID_qE(M)$ if
$\delta_R^{(k_C)}(Y)=A_kY, \ \mbox{for all} \  k \in \mathbb{N}$.\\
The ring $R$ is called an \textbf{ iterative $q$-difference
Picard-Vessiot ring} for $ID_qE(M)$ ($IPV_q$-ring for short) if it
fulfills the following conditions :
\begin{enumerate}
\item $R$ is a simple $ID_q$-ring (that means that
$ R$ contains no proper iterative $q$-difference ideal ),
\item $ID_qE(M)$ has a fundamental solution matrix $Y$ with coefficients
in $R$,
\item $R$ is generated by the coefficients of $Y$ and $det(Y)^{-1}$,
\item $C(R)=C(L)$.
\end{enumerate}
\end{defin}
\begin{remark}[Classical case]
As in Remark \ref{remark:eqcl}, we easily see that if $q$ is not a
root of unity, the notion of an iterative Picard-Vessiot ring is
exactly the same as the notion of Picard-Vessiot ring in the sense
of Singer, van der Put (\cite{VPS2} $1.1$).
\end{remark}

\begin{prop}\label{prop:const}
Let  $(L,\delta_L^*)$ be an iterative $q$-difference field, with
algebraically closed field of constants $C(L)$, and let $R/L$ be a
simple $ID_q$-ring. Then $R$ is a reduced $ID_q$-ring. Moreover, if
$R$ is finitely generated over $L$, we have $C(L)=C(E)$ where $E$
denotes the localization of $R$ by its set of non zeros divisors.
\end{prop}

\begin{proof}
The fact that $R$ is a reduced $ID_q$-ring is a consequence of Lemma
\ref{lemma:rad} where it is shown that if  $I$ is an $ID_q$-ideal
the same is true for its radical. For the second statement, let us
assume that $R$ is finitely generated over $L$. Let $c$ be a non
zero constant of $E$ and put $J=\lbrace a \in R| a.c \in R \rbrace$.
First of all, because $\delta_E^{(1)}=\frac{\sigma_q -id}{(q-1)t}$,
we have that $\sigma_q^k(c)=c$ for all  $k \in \mathbb{N}$. It is
then quite clear that $J$ is an $ID_q$-ideal of $R$ because of
$\delta_R^{(
k)}(a.c)=\sigma_q^k(c).\delta_R^{(k)}(a)=c.\delta_R^{(k)}(a)$ for
all $k \in \mathbb{N}$. Since $R$ is simple, and $J$ is a non
trivial $ID_q$-ideal, we have $J=R$, and thus $1.c=c \in R$. Suppose
that $c \notin C(L)$. Thus for all $d \in C(L)$ the ideal $(c-d)R$
is a non trivial $ID_q$-ideal in $R$ and also equal to $R$. This
means
that $(c-d) \in R^*$ for all $d \in C(L)$.\\
Let $\bold{\phi}_c: \ Spec(R) \mapsto \mathbb{A}^1_L$ be the
morphism induced by $$\xymatrix{ \phi : & L[T] \ar[r] & R,
                                          & T \ar@{|->}[r] & c .}$$
Since $Im(\bold{\phi_c)}\cap \mathbb{A}^1_L(C(L))$ is empty,
$Im(\bold{\phi}_c)$ does not contain any open subset of
$\mathbb{A}^1_L$. Therefore the image of  $\bold{\phi}_c$ in
$\mathbb{A}^1_L$ is finite and closed. This implies that $c$ is
algebraic over $L$. Let $P \in L[X]$ be the minimal monic polynomial
annihilating $c$. We have
$\delta_L^{(k)}(P(c))=P^{\delta_L^{(k)}}(c)=0$ where
$P^{\delta_L^{(k)}}$ denotes the element of $L[X]$ obtained from $P$
by applying $\delta_L^{(k)}$ on the coefficients of $P$. By
minimality of $P$ we conclude that $P\in C(L)[X]$. Because $C(L)$ is
algebraically closed, we then have $c \in C(L)$. This is a
contradiction!\end{proof}

\begin{prop}\label{prop:matsol}
Let $(L,\delta_L^*)$ be an $ID_q$-field and $(R,\delta_R^*)$ be an
$ID_q$-ring with $q$-difference operator extending the one given on
$L$. Let $Y$ and $\tilde{Y}$ be two elements of $Gl_n(R)$,
fundamental matrices of solutions for the $ID_qE$,
$\delta_R^{(k_C)}(\bold{y})=A_k\bold{y}$. Then, there exists a
 matrix $P \in Gl_n(C(R))$ such that $\tilde{Y}=YP$. Moreover, if
 both $L$ and $R$ satisfy the conditions of Proposition
 \ref{prop:const} then $P \in Gl_n(C(L))$.
 \end{prop}
 \begin{proof}

 It is obvious that there exists $P \in Gl_n(R)$ such that
 $\tilde{Y}=YP$. We want to show by induction that for all $k \in
 \mathbb{N}^*$, we have $\delta_R^{(k)}(P)=0$. For $k=1$ we obtain  $$\delta_R^{(1)}(\tilde{Y})=\delta_R^{(1)}(Y) P + \sigma_q(Y)
 \delta_R^{(1)}(P)=A_1\tilde{Y} + \sigma_q(Y) \delta_R^{(1)}(P).$$
Thus, $\delta_R^{(1)}(P)=0$ (because $\sigma_q$ is an automorphism
of $Gl_n(R)$). Using the formula
$$\delta_R^{(k)}(\tilde{Y}) = \sum_{i+j=k} \sigma_q^i(\delta_R^{(j)}(Y))
\delta_R^{(i)}(P),$$ we get by induction that $\delta_R^{(k)}(P)=0$
for all $k \in \mathbb{N}^*$. This implies that $ P \in Gl_n(C(R))$.
\end{proof}

\begin{theorem}\label{theorem:pv}
Let $(L,\delta_L^*)$ be an iterative $q$-difference field with
$C(L)$ algebraically closed and let $(M,\delta_M^*)$ be an object of
$IDM_q(L)$ with iterative $q$-difference equation
$\delta_L^{(k_C)}(\bold{y})=A_k\bold{y}$ ($ID_qE(M)$). Then there
exists an iterative $q$-difference Picard-Vessiot ring for the
iterative $q$-difference equation which is unique up to iterative
$q$-difference isomorphism.
\end{theorem}
\begin{proof}
Let $m$  be the dimension of $M$ over $L$  and set
$U=L[x_{(i,j)},det(x_{(i,j)})^{-1}]$. The algebra
 $U_0:=L[x_{(i,j)}]$ is given a structure of $q$-difference extension of $L$  via     $\sigma_q(X):=\frac{A_1 }{(q-1)t}X
+X$ where $X=(x_{(i,j)})_{(i,j)}$. Because $\sigma_q$ is a
ring-automorphism, we have that the ideal $S$ generated in $U_0$ by
$det(x_{i,j})$ is a $\sigma_q$-ideal and a multiplicatively closed
set. $U_0$ has  a non trivial  $ID_q$-structure via
$$\delta_{U_0}^{*}:= \  \delta_P^{(k_C)}(X)=A_k X, \ \mbox{for all} \  k \in \mathbb{N}.$$
Because $S$ satisfies the condition of Proposition \ref{prop:ext},
there exists a unique iterative $q$-difference operator
$\delta_{S^{-1}U_0}^*$ extending $\delta_{U_0}^*$ on $U= S^{-1}U_0$.
Let $P \subset U$ be a maximal $ID_q$-ideal of $U$. Then $R:=U/P$ is
a simple $ID_q$-ring and $Y:= \overline{X}$, the image of $X$ under
the projection of $U$ to $R$,  is a fundamental solution matrix of
$ID_qE(M)$. Moreover $R/L$ is generated by the coefficients of $Y$
and $det(Y)^{-1}$. Thus $R$ is an iterative $q$-difference Picard-Vessiot ring.\\
Assume that $(R_1,\delta_{R_1}^*)$ and $(R_2,\delta_{R_2}^*)$ are
two iterative $q$-difference Picard-Vessiot rings for $M$ with
fundamental solution matrix $Y_1$ (resp. $Y_2$) in $R_1$ (resp.
$R_2$). Put $N=R_1 \otimes_L R_2$. As in Proposition \ref{prop:tens}
we endow $N$ with an $ID_q$-structure. Let $P \subset N$ be a
maximal $ID_q$-ideal, then $R':=N/P$ is a simple $ID_q$ ring. The
two maps :
$$\xymatrix{ \phi_1: & R_1 \ar[r] & R',
      & r_1 \ar@{|->}[r] & \overline{(r_1\otimes 1)}}$$
and
$$\xymatrix{ \phi_2: & R_2 \ar[r] & R',
      & r_2 \ar@{|->}[r] & \overline{(1\otimes r_2)}}.$$
induced by the natural inclusions are $ID_q$-monomorphisms, and
$\phi_1(Y_1)$ and $\phi_2(Y_2)$ are two fundamental matrix solutions
for $M$ in $R'$. By Proposition \ref{prop:matsol}, there exists $P
\in Gl_n(C(L))$ such that $\phi_1(Y_1)=\phi_2(Y_2)P$
($C(L)=C(R_1)=C(R_2)=C(R')$), which implies that $\phi_1(R_1) \simeq
\phi_2(R_2)$. This concludes the proof.
\end{proof}

\subsection{The iterative $q$-difference Galois group}
In this  section, we will define the iterative $q$-difference Galois
group associated to an iterative $q$-difference module. The way of
describing such a group is the exact translation in the
$q$-difference world of the work of A. Roescheisen (see \cite{Ro})
in the case of iterative differential Galois theory. Until the end
of this section, $(L,\delta_L^*)$ will be an iterative
$q$-difference field with  algebraically closed field of constants
$C$, $(R, \delta_R^*)$  an iterative $q$-difference Picard-Vessiot
ring for the iterative $q$-difference equation $\lbrace
\delta_L^{(k_C)}Y=A_k Y, \ k \in \mathbb{N}\rbrace$ defined over
$L$.
\begin{Not}
Let $S$ be a ring. We denote by $Loc(S)$ its localization by its set
of non-zero divisors.
\end{Not}

\subsubsection{Functorial definition}
First of all, let us remark that, given an algebra $A$ over $C$ and
an iterative $q$-difference ring $(S, \delta_S^*)$, we  define an
iterative $q$-difference operator on $S \otimes_C A$ by setting
$\delta_{S\otimes_C A}^{(k)}(s \otimes f):= \delta_S^{(k)}(s)
\otimes f$ for all $k \in \mathbb{N}$. As in \cite{Ro} Definition $10.4$, we say that
$\delta_S^*$ is \textbf{extended trivially} to $S\otimes_C A$.

\begin{defin}\label{defin:galois}
 Let us define the functor
$$\xymatrix{
\underline{Aut}(R/L): & (Algebras/C) \ar[r] & (Groups), & A
\ar@{|->}[r] & Aut_{ID_q}(R \otimes_C A / L \otimes_C A) }$$ where
$\delta_R^*$ (resp. $\delta_L^*$) is extended trivially to $R
\otimes_C A$ (resp. $L \otimes_C A$).
\end{defin}
In the following, we will show that the functor
$\underline{Aut}(R/L)$ is representable by a certain $C$-algebra of
finite type and hence is an affine group-scheme of finite type over
$C$.

\begin{lemma}\label{lemma:idea}
Let  $R$ be a simple $ID_q$-ring with $C(R)=C$, let $A$ be a
finitely generated $C$-algebra and $R_A:=R\otimes_C A $ with
$ID_q$-structure trivially extended from $R$. Then there is a
bijection
$$\xymatrix{
\mathcal{I}(A) \ar[r] & \mathcal{I}_{ID_q}(R_A)\ar[l],\\
 I \ar@{|->}[r] & R_A(1\otimes_C I)=R \otimes_C I,\\
 J\cap(1\otimes_C A) & J \ar@{|->}[l] }$$
 between the ideals of $A$ and the $ID_q$-ideals of $R_A$.

\end{lemma}

\begin{proof}
Obviously, the two maps are well defined, and  we only have to prove
that they are inverse to each other.
\begin{enumerate}
\item We will prove that for $I \in \mathcal{I}(A)$, we have $(R \otimes_C
I) \cap (1\otimes_C A) =I$. It is obvious that $I$ is contained in
the ideal on  the left side. Now let us consider   a $C$-basis
$\lbrace e_i| i \in \tilde{N} \rbrace$ of $I$ ; then $R \otimes_C I$
is a free $R$-module with  basis $\lbrace 1\otimes e_i| i \in
\tilde{N} \rbrace$  and an element $f=\sum_{i \in \tilde{N}} r_i
\otimes e_i \in R \otimes_C I$ is  constant if and only if all the
$r_i$'s are constants, i.e., if $f \in I$.
\item Conversely we have to prove that for $J \in
\mathcal{I}_{ID_q}(R_A)$, we have $R \otimes_C J\cap(1\otimes_C
A)=J$. It is clear that $J$ contains the ideal on the left side.
Now, let $\lbrace e_i| i \in N \rbrace$ a $C$-basis of $A$, where
$N$ denotes an index set. Then, $\lbrace 1 \otimes e_i| i \in N
\rbrace$ is also a basis
for the free $R$-module $R_A$.\\
For any subset $N_0$ of $N$ and $i_0 \in N_0$, let $Ann_{N_0,i_0}$
be the ideal   of all $r \in R$ such that there exists an element
$g=\sum_{i \in N_0} s_i \otimes e_i \in J$ with $s_{i_0}=r$. Since
the iterative $q$-difference operator of $R_A$ acts trivially on $A$
and $J$ is an $ID_q$-ideal, it is clear that $Ann_{N_0,i_0}$ is an
$ID_q$-ideal. Because $R$ is
simple, $Ann_{N_0,i_0}$ is equal to $(0)$ or $R$.\\
Now, let $N_0 \subset N$ be  minimal for  the property that
$Ann_{N_0,i_0}\neq (0)$ for at least one index $i_0 \in N_0$
(minimal in the lattice of subsets). So there exists $g=\sum_{i \in
N_0} s_i \otimes e_i \in J$ with $s_{i_0}=1$ and by minimality of
$N_0$ we conclude that for all $k \in \mathbb{N}^*$,
$\delta^{(k)}(g)=\sum_{i \in N_0, i \neq i_0} \delta_R^{(k)}(s_i)
\otimes e_i=0$. This implies $g \in J\cap(1\otimes_C A)$. Now let
$g=\sum_{i \in N} s_i \otimes e_i \in J$ be an arbitrary element and
denote by $N_1$ the set of indices $i$ with $s_i \neq 0$. It follows
from the definition that $Ann_{N_1,i} \neq (0)$ for all $i \in N_1$.
Hence there exists $N_0 \subset N_1$ minimal as above, $i_0 \in N_0$
and $f=\sum_{i \in N_0} r_i \otimes e_i \in J\cap(1\otimes_C A)$
with $r_{i_0}=1$. By induction on the cardinality of $N_1$, we may
assume that $g-s_{i_0}f \in R \otimes_C J\cap(1\otimes_C A) \subset
J$. Therefore $g=g-s_{i_0}f + s_{i_0}f \in  R \otimes_C
J\cap(1\otimes_C A)$ and hence $  R \otimes_C J\cap(1\otimes_C
A)=J$. \end{enumerate}
\end{proof}

\begin{prop}\label{prop:carac}
Let $R/L$ be an  iterative $q$-difference Picard-Vessiot ring
associated to an iterative $q$-difference equation and let $T$ be a
$ID_q$-simple ring containing $L$ with $C(T)=C=C(L)$ such that there
exists a  fundamental matrix of solutions $Y \in Gl_n(T)$. Then
there exists a finitely generated $C$-algebra $U$ (with trivial
$ID_q$-structure) and a $T$-linear $ID_q$-isomorphism
$$\xymatrix{ \gamma_T :T \otimes_L R \ar[r] & T \otimes_C U.}$$
where the $ID_q$-structure is extended trivially to  $T \otimes_C
U$.\\
 (Actually $U$ is isomorphic to the ring of constants of $T
\otimes_L R$.)
\end{prop}

\begin{proof}
$R$ is obtained as a quotient of $L[X_{i,j},(det(X))^{-1}]$ with
iterative $q$-difference operator given by $\delta^{(k)}(X)=A_k X \
\mbox{for all} \ k \in \mathbb{N}$  by a maximal $ID_q$-ideal $P
\subset L[X_{i,j},(det(X))^{-1}]$. We then define a $T$-linear
homomorphism
$$\xymatrix{
\gamma_T : T \otimes_L L[X_{i,j},det(X)^{-1}] \ar[r] & T \otimes_C
C[Z_{i,j},det(Z)^{-1}]}$$ by  $X_{i,j}\mapsto \sum_{k=1}^n Y_{i,k}
\otimes Z_{k,j}$. The morphism $\gamma_T$ is indeed a $T$-linear
isomorphism and if we extend the $ID_q$-structure trivially to
$L[Z_{i,j},(det(Z))^{-1}]$, $\gamma_T$  induces an
$ID_q$-isomorphism.\\
By the previous lemma, the $ID_q$-ideal $\gamma_T(T \otimes P)$ is
equal to $T \otimes I$ for an ideal $I \subset
C[Z_{i,j},(det(Z))^{-1}]$. So for $U:= C[Z_{i,j},(det(Z))^{-1}]/I$,
$\gamma_T$ induces an $ID_q$-isomorphism
$$\xymatrix{ \gamma_T : T\otimes_L R \ar[r] & T \otimes_C U.}$$
\end{proof}

\begin{theorem}\label{theorem:fix}
Let $R/L$ be an iterative $q$-difference  Picard-Vessiot ring. Then
 the group functor $\underline{Aut}(R/L)$ is representable by the
finitely generated  $C$-algebra $U=C(R\otimes_L R)$, i.e.,
$\underline{Aut}(R/L)$ is an affine group-scheme of finite type over
$C$.

\end{theorem}

\begin{defin}

We call the affine group scheme $\underline{Aut}(R/L)$ the
\textbf{Galois group scheme} $\underline{Gal}(R/L)$   of $R$ over
$L$.
\end{defin}

\begin{proof}
\textbf{Proof of theorem \ref{theorem:fix}}\\
First we will show that for every $C$-algebra $A$  any $L_A$-linear
$ID_q$-homomorphism \\
$\xymatrix{f:R_A \ar[r] & R_A}$ is an isomorphism. The kernel of
such a homomorphism $f$ is an $ID_q$-ideal of $R_A$. So by Lemma
\ref{lemma:idea}, it is generated by constants, i.e., elements in
$1\otimes A$. But $f$ is $A$-linear so its kernel is zero. If $X \in
Gl_n(R)$ is a fundamental solution matrix, then $f(X)\in Gl_n(R_A)$
is also a fundamental solution matrix and so there exists a matrix
$D\in Gl_n(C_{R_A})=Gl_n(A)$ such that $X=f(X)D=f(XD)$. Hence
$X_{i,j},det(X)^{-1} \in Im(f)$
 and since $R$ is generated
by $X_{i,j},det(X)^{-1}$  over $L$, the homomorphism $f$ is also surjective.\\
Using the isomorphism $\gamma:=\gamma_R$ of Proposition
\ref{prop:carac}, for a $C$-algebra $A$, we obtain a chain of
isomorphisms
$$\begin{array}{ccc}
Aut^{ID_q}(R_A/L_A) &=  Hom_{L_A}^{ID_q}(R_A,R_A)& \simeq  Hom_{R_A}^{ID_q}(R_A \otimes_L R,R_A)\\
 &   & \\
  \simeq  Hom_{R_A}^{ID_q}(R_A \otimes_C U,R_A)& \simeq
Hom_{C}^{ID_q}(U,R_A) & \simeq  Hom_{C}(U,A).
\end{array}$$
Hence $U$  represents the functor $\underline{Aut}(R/L)$.
\end{proof}

\begin{remark}\label{remark: action}
By taking a closer look on the isomorphisms in the previous proof,
we see that the universal object $id_U \in Hom_C(U,U)$ corresponds to
the $ID_q$-automorphism $\xymatrix{\rho \otimes id_U: R\otimes_C
U \ar[r] & R\otimes_C U}$ where $\xymatrix{\rho=\gamma_R \circ(1
\otimes id_R): R \ar[r] & R \otimes_L R \ar[r] & R \otimes_C U}$.
Therefore the action of $g \in \underline{Aut}(R/L)(A)=Hom_C(U,A)$ on $r \in R$ is given by
$$g.r=(id_R \otimes g)(\gamma_R(1 \otimes r)) \in R \otimes_C A.$$
\end{remark}

\begin{coro}\label{coro:tor}
Let $R/L$ be an iterative $q$-difference  Picard-Vessiot ring over
$L$ and $\mathcal{G}:=\underline{Gal}(R/L)$ the Galois group scheme
of $R$. Then $Spec(R)$ is a $\mathcal{G}_L$-torsor.
\end{coro}
\begin{proof}
The isomorphism $\gamma:=\gamma_R$ of proposition \ref{prop:carac},
determines an isomorphism of schemes
$$\xymatrix{Spec(\gamma):Spec(R)\times_L \mathcal{G}_L =Spec(R) \times_C
\mathcal{G} \ar[r] & Spec(R) \times_L Spec(R).}$$ By  the previous
remark and $R$-linearity of $\gamma$, the composition of
$Spec(\gamma)$
 with the projection on the second factor  gives the action of $\mathcal{G}_L$ on
 $Spec(R)$ and the composition with the  projection on the first factor
 equals
  the map  $Spec(R) \times_L \mathcal{G}_L \rightarrow
  Spec(R)$. In other words, $Spec(R)$ is a $\mathcal{G}_L$-torsor.
\end{proof}
\subsubsection{Galois correspondence}

\begin{prop}[Structure of the iterative $q$-difference
ring]\label{prop:prod}
 Let $R/L$ be an iterative $q$-difference
Picard-Vessiot ring over $L$. Then, there exist idempotents
$e_1,..,e_s \in R$ such that
\begin{enumerate}
\item $R=R_1 \oplus...\oplus R_s$ where $R_i=e_iR$ and is a domain,
\item The direct sum $E$ of the fraction fields of the $R_i$'s is an
iterative $q$-difference ring. $E$ is  called \textit{the total
iterative $q$-difference Picard-Vessiot extension of $R$}.
\end{enumerate}
\end{prop}
\begin{proof}
Here, we give a partial analogue of Corollary $1.16$ of \cite{VPS2}.
We will thus follow the proof of Singer, van der Put. But because we
work in any characteristic, it will be necessary to
appeal to the book of Demazure, Gabriel (\cite{Dem}) to assure smoothness.\\
Let $\overline{L}$ be an algebraic closure of $L$ and
$R=O(\mathcal{Z})$ for some $\mathcal{G}_L$-torsor $\mathcal{Z}$.
Since $\mathcal{G}_L(\overline{L})$ acts transitively on
$\mathcal{Z}(\overline{L})$, this latter algebraic subset must be
smooth (\cite{Dem} $4.2$). Therefore the $L$-irreducible components
$\mathcal{Z}_1,...,\mathcal{Z}_s$ must be disjoint. Thus
$O(\mathcal{Z})$ is equal to the product of the integral domains
$R_i=O(\mathcal{Z}_i)$. Now let us consider the set $S$  of non zero
divisors in $R$. It is a multiplicatively closed set which does not
contain $0$, stable under the action of $\sigma_q$. By Proposition
\ref{prop:ext}, the ring $RS^{-1}$ is endowed with an iterative
$q$-difference structure and it is obvious that
$RS^{-1}=\bigoplus_{i=1}^s Frac(R_i)$ where $Frac(R_i)$ denotes the
fraction field of $R_i$.
\end{proof}

The next proposition shows that to be  a torsor for an $ID_q$-simple
ring means, roughly speaking, to be an iterative $q$-difference
Picard-Vessiot ring.

\begin{prop}\label{prop:caracpi} Let  $R/L$ be  a simple
$ID_q$-ring with algebraically closed field of constants $C(R)=C$.
Further let $\mathcal{G} \subset Gl_{n,C}$ be an affine group scheme
over $C$. Assume that $Spec(R)$ is a $\mathcal{G}_L$-torsor such
that the corresponding isomorphism $\gamma : R\otimes_L R
\rightarrow R \otimes_C C[\mathcal{G}]$ is an $ID_q$-isomorphism.
Then $R$ is an iterative $q$-difference Picard-Vessiot ring over
$L$.
 \end{prop}

\begin{proof}
Since $Spec(R)$ is a $\mathcal{G}_L$-torsor, the fiber product
$Spec(R)\times_{\mathcal{G}_L}Gl_{n,L}$ is a $Gl_{n,L}$-torsor.\\
$Spec(R)\times_{\mathcal{G}_L}Gl_{n,L}$ is obtained as the quotient
of the direct  product by the $\mathcal{G}_L$-action given by
$(x,h).g:=(xg,g^{-1}h)$ and is a right $Gl_{n,L}$-scheme acting on
the second factor. By Hilbert's Theorem $90$, every
$Gl_{n,L}$-torsor is trivial, i.e., we have an
$Gl_{n,L}$-equivariant isomorphism
$$\xymatrix{Spec(R)\times_{\mathcal{G}_L}Gl_{n,L} \ar[r] &
Gl_{n,L}}.$$ Then the closed embedding $\xymatrix{Spec(R) \ar[r] &
Spec(R)\times_{\mathcal{G}_L}Gl_{n,L} \ar[r] & Gl_{n,L}}$ leads to
an epimorphism $\xymatrix{L[X_{i,j},(det(X))^{-1}] \ar[r] & R}$,
which is $\mathcal{G}_L $-equivariant. Denote the image of $X$ by
$Y$. Then we obtain  that the action of $\mathcal{G}$ on $Y$ is
given by $Y \mapsto Yg$ for any $L$-valued point  $g \in
\mathcal{G}_L(L)$. Since by assumption for every $C$-algebra $A$
with trivial $ID_q$-structure, the action of $\mathcal{G}(A)$
commutes with the iterative $q$-difference operator
$\delta^{(k)}(Y).Y^{-1}$ is $\mathcal{G}$-invariant for all $k \in
\mathbb{N}$. So $\delta^{(k)}(Y).Y^{-1} =A_k$ belongs to $ Gl_n(L)$
and $Y$ is a fundamental solution matrix for the equation $\lbrace
\delta^{(k)}(Y).Y^{-1} \rbrace_{k \in \mathbb{N}}$. Hence $R$
is an $ID_q$-Picard-Vessiot ring.\end{proof}

In order to get a convenient Galois correspondence, we are obliged
to define the notion of an invariant in a functorial way. Let $S$ be
a $C$-algebra and $\mathcal{H}/C$ be a subgroup functor of the
functor $\underline{Aut}(S/C)$, i.e., for every $C$-algebra $A$, the
set $\mathcal{H}(A)$ is a group acting on $S_A$ and this action is
functorial. An element $s \in S$ is called invariant if for all $A$,
the element $s\otimes 1 \in S_A$ is invariant under
$\mathcal{H}(A)$. The ring of invariants is denoted by
$S^{\mathcal{H }}$. Let $E=Loc(S)$ be the localization of $S$ by all
non zero-divisors. We call an element $e=\frac{r}{s} \in E$
invariant under $\mathcal{H}$, if for each $C$-algebra $A$ and all
$h \in \mathcal{H}(A)$,

$$h.(r\otimes 1 ).(s \otimes 1)=(r \otimes 1).h.(s \otimes 1).$$
 $E^{\mathcal{H}}$ denotes the ring of invariants (for the
independence of this definition of the choice of representation of
$e$ see \cite{Ro} section $11$ or \cite{Jan},$I.2.10$).\\

\begin{lemma}
Let $R/L$ be an iterative $q$-difference  Picard-Vessiot ring over
$L$, let $E$ denote its total iterative $q$-difference
Picard-Vessiot extension and  $\mathcal{G}:=\underline{Gal}(R/L)$
the Galois group scheme of $R$. Let $\mathcal{H} \subset
\mathcal{G}$ be a closed subgroup-scheme. Denote by
$\xymatrix{\pi_{\mathcal{H}}^{\mathcal{G}} :C[\mathcal{G}]\ar[r] &
C[\mathcal{H}]}$ the epimorphism corresponding to the inclusion
$\xymatrix{\mathcal{H} \ar@{^{(}->}[r] & \mathcal{G}}$. Then an
element of $ \frac{r}{s} \in E$ is invariant under the action of
$\mathcal{H}$ if and only if $r\otimes s -s \otimes r$ is in the
kernel of the map

$$\xymatrix{(id_R \otimes \pi_{\mathcal{H}}^{\mathcal{G}}) \circ \gamma :R
\otimes_L R \ar[r] & R \otimes_C C[\mathcal{H}].}$$
\end{lemma}
\begin{proof}
An element $\frac{r}{s} \in E$ is invariant under the action of $\mathcal{H}$ if and only if it is invariant under the universal element in $\mathcal{H}$, namely $ \pi_{\mathcal{H}}^{\mathcal{G}} \in \mathcal{G}(C[\mathcal{H}])$. By remark \ref{remark: action} and $R$-linearity of $\gamma$, we have
$$(id_R  \otimes  \pi_{\mathcal{H}}^{\mathcal{G}})(\gamma(r \otimes s))=(r \otimes 1).\pi_{\mathcal{H}}^{\mathcal{G}}(s \otimes 1) \in R \otimes_CC[\mathcal{H}].$$
Therefore $r\otimes s -s \otimes r$ is in the considered Kernel if and only if $\frac{r}{s}$ is invariant under $\mathcal{H}$.
\end{proof}

\begin{prop}\label{prop:inva}
For every closed subgroup scheme $\mathcal{H} \subset \mathcal{G}$,
the ring $E^{\mathcal{H}}$ is an $ID_q$-ring in which every non zero
divisor is a unit. Furthermore we have $E^{\mathcal{H}}=L$ if and
only if $\mathcal{H}=\mathcal{G}$.
\end{prop}
\begin{proof}
By the previous lemma, it is obvious that $E^{\mathcal{H}}$ is an
$ID_q$-ring in which every non-zero divisor is a unit. Next, let
$\frac{r}{s} \in E^{\mathcal{H}}$. Then for all $ k \in \mathbb{N}$,
we have
$$ \delta^{k}(r\otimes s -s \otimes r).(s^k \otimes s^k) =$$
$$\sum_{i_1+i_2+i_3=k} \sigma_q^{i_1+i_3} (\delta^{(i_2)}(\frac{r}{s}))s^k
\sigma_q^{i_3}(\delta^{(i_1)}(s))\otimes \delta^{(i_3)}(s)s^k -
\delta^{(i_1)}(s)s^k \otimes
\sigma_q^{i_1+i_3}(\delta^{(i_2)}(\frac{r}{s}))s^k \sigma_q^{i_1
}(\delta^{(i_3)}(s))$$ $$=\sum_{i_1+i_2+i_3=k}
(\sigma_q^{i_3}(\delta^{(i_1)}(s))
\otimes\delta^{(i_3)}(s))(\sigma_q^{i_1+i_3}(\delta^{(i_2)}(\frac{r}{s}))
s^k \otimes s^k) - $$ $$ \sum_{i_1+i_2+i_3=k} (\delta^{(i_1)}(s)
\otimes\sigma_q^{i_1}(\delta^{(i_3)}(s)))(     s^k
\otimes\sigma_q^{i_1+i_3}(\delta^{(i_2)}(\frac{r}{s})) s^k)=$$
$$\sum_{i+j=k} (\delta^{(i)}(s \otimes
s))(\sigma_q^i(\delta^{(j)}(\frac{r}{s}))s^k \otimes s^k -s^k
\otimes \sigma_q^i(\delta^{(j)}(\frac{r}{s}))s^k).
$$
The left hand side lies in $Ker(id_R \otimes
\pi_{\mathcal{H}}^{\mathcal{G}})$, since this kernel is an
$ID_q$-ideal . So by induction, we get that $(s\otimes
s)(\delta^{(k)}(\frac{r}{s})s^k \otimes s^k -s^k \otimes
\delta^{(k)}(\frac{r}{s})s^k)\in Ker(id_R \otimes
\pi_{\mathcal{H}}^{\mathcal{G}}) $ and hence
$\delta^{(k)}(\frac{r}{s}) \in E^{\mathcal{H}}$. \\
For the second statement : if $\mathcal{H}=\mathcal{G}$, then
$\pi_{\mathcal{H}}^{\mathcal{G}} =id_{C[\mathcal{G}]}$ and the
considered kernel is trivial. Hence $r\otimes s = s \otimes r \in
R\otimes_L R$ is trivial for all $\frac{r}{s} \in E^{\mathcal{G}}$.
Thus, there exists $c \in L$ such that $r=cs$, i.e., $\frac{r}{s} =c
\in L$.\\
Assume $\mathcal{H} \subsetneq \mathcal{G}$.Since
$\mathcal{Z}=Spec(R)$ is a $\mathcal{G}_L$-torsor, the quotient
scheme $\mathcal{Z}/\mathcal{G}_L$ is equal to $Spec(L)$, in
particular it is a scheme, and since $\mathcal{G}_L$ and
$\mathcal{H}_L$ are affine, $\mathcal{G}_L/\mathcal{H}_L$ also is a
scheme. So by \cite{Jan},$I.5.16.(1)$, $\mathcal{Z}/\mathcal{H}_L
\simeq \mathcal{Z}\times^{\mathcal{G}_L}
(\mathcal{G}_L/\mathcal{H}_L)$ is a scheme. According to Proposition
\ref{prop:prod}, $\mathcal{Z}$ is equal to the disjoint union of its
irreducible components $\lbrace \mathcal{Z}_i \rbrace _{i=1,...,s}$.
Let  $pr: \mathcal{Z} \mapsto \mathcal{Z}/\mathcal{H}_L$ denote the
canonical projection. Now let $\overline{U} \subseteq
\mathcal{Z}/\mathcal{H}_L$ be an  affine open subset such that its
inverse image $\mathcal{U}$ by $pr$ has a non empty intersection
with all the $\mathcal{Z}_i$. We have a monomorphism $pr_*
:\mathcal{O}_{\mathcal{Z}/\mathcal{H}_L}(\overline{U}) \rightarrow
\mathcal{O}_{\mathcal{Z}}(\mathcal{U})$ whose image is
$\mathcal{O}_{\mathcal{Z}}(\mathcal{U})^{\mathcal{H}}$. By
construction of $\overline{U}$, we have
$\mathcal{O}_{\mathcal{Z}}(\mathcal{U})^{\mathcal{H}} \subset
E^{\mathcal{H}}$. If $E^{\mathcal{H}}=L$, then also
$\mathcal{O}_{\mathcal{Z}}(\mathcal{U})^{\mathcal{H}}=L$. So, for
every  affine open subset  $\overline{U} \subseteq
\mathcal{Z}/\mathcal{H}_L$  such that its inverse image
$\mathcal{U}$ by $pr$ has a non empty intersection with all the
$\mathcal{Z}_i$, we have
$\mathcal{O}_{\mathcal{Z}/\mathcal{H}_L}(\overline{U})=L$, i.e.,
$\overline{U}\simeq Spec(L)$ is a single point. Hence $\mathcal{Z}/
\mathcal{H}_L = Spec(L)$, which contradicts the assumption
$\mathcal{H} \subsetneq \mathcal{G}$.
\end{proof}

\begin{theorem}[Galois correspondence]\label{theorem:galcor}
Let $R/L$ be an iterative $q$-difference  Picard-Vessiot ring over
$L$, let $E$ denotes its total iterative $q$-difference
Picard-Vessiot extension and let $\mathcal{G}:=\underline{Gal}(R/L)$
be the Galois group scheme of $R$. \begin{enumerate} \item Then
there is an anti-isomorphism of lattices between:
$$\mathfrak{H}:=\lbrace \mathcal{H}| \mathcal{H}\subset \mathcal{G}
\ \mbox{closed subgroup scheme of} \ \ \mathcal{G} \rbrace$$ and

$$\mathfrak{T}:= \lbrace T| L \subset T \subset E \ \ \mbox{intermediate} \ ID_q- \mbox{ring
s.t. any non zero divisor of T is a unit of T}\rbrace$$ given by
$\Psi: \mathfrak{H} \rightarrow \mathfrak{T}$, $\mathcal{H} \mapsto
E^{\mathcal{H}}$ and
$\Phi: \mathfrak{T} \rightarrow  \mathfrak{H}$, $T \mapsto \underline{Gal}(RT/T)$.\\
\item If $\mathcal{H} \subset \mathcal{G}$ is normal then
$R^{\mathcal{H}}$ is  an iterative $q$-difference  Picard-Vessiot
ring over $L$ and  $E^{\mathcal{H}}$ is  its total iterative
$q$-difference Picard-Vessiot extension; the Galois group scheme of
$R^{\mathcal{H}}$ over $L$ is isomorphic to $\mathcal{G}/
\mathcal{H}$.\\
\item For $\mathcal{H} \in \mathfrak{H}$, the extension
$E/E^{\mathcal{H}}$ is separable if and only if $\mathcal{H}$ is
reduced.
\end{enumerate}
\end{theorem}

\begin{proof}
\begin{enumerate}

\item Let $T \in \mathfrak{T}$ be an intermediate $ID_q$ ring such that
any non zero divisor of T is a unit of T. Then the compositum $RT
\subset E$ is a $ID_q$-Picard-Vessiot ring over $T$. Furthermore,
the canonical $ID_q$-epimorphism $RT \otimes_C C[\mathcal{G}]
\mapsto RT \otimes_T RT$ gives rise to an $ID_q$-epimorphism
$$ \xymatrix{
RT \otimes_C C[\mathcal{G}] \ar[r]^{ \gamma_{RT}^{-1}} & RT\otimes L
R \ar[r] & RT \otimes_T RT}.$$

By Lemma \ref{lemma:idea}, the kernel of this epimorphism is given
by $RT \otimes_C I$ for some ideal $I \subset C[\mathcal{G}]$.
Denote by $\mathcal{H}$ the closed sub-scheme of $\mathcal{G}$
defined by $I$, then $\gamma_{RT}$ induces a isomorphism
$$RT \otimes_T RT \simeq RT \otimes_C C[\mathcal{H}].$$
By construction, this isomorphism is the isomorphism for  the base
ring $T$, hence the sub-scheme $\mathcal{H}$ equals the Galois group
scheme $\underline{Gal}(RT/T)$. Thus $\underline{Gal}(RT/T)$ is
indeed a closed subgroup scheme  of $\mathcal{G}$.\\
Now let us apply Proposition \ref{prop:inva} to the extension $E/T$.
It follows that $E^{\underline{Gal}(RT/T)}=T$, so $\Psi \circ
\Phi=id_{\mathfrak{T}}$. On the other hand, for given $\mathcal{H}
\in \mathfrak{H}$ and $T:=E^{\mathcal{H}}$, we get an $ID_q$-
epimorphism $RT \otimes_T RT \mapsto RT \otimes_C C[\mathcal{H}]$
induced by $\gamma_{RT}$. This embeds  $\mathcal{H}$ as a closed
subgroup scheme in  $\underline{Gal}(RT/T)$. But the localization
$Loc(RT)$ of $RT$ by its set of non zero divisors is equal to $E$,
so $Loc(RT)^{\mathcal{H}}=E^{\mathcal{H}}=T$ and so by Proposition
\ref{prop:inva}, we have $\mathcal{H}=\underline{Gal}(RT/T)$.
Thereby $\Phi \circ \Psi =id_{\mathfrak{H}}$.
\item Let $\mathcal{H} \subset \mathcal{G}$ be normal. The
isomorphism $\gamma$ is $\mathcal{H}$-equivariant and hence we get
an $ID_q$-isomorphism
$$R \otimes_L R^{\mathcal{H}} \simeq R \otimes_C
C[\mathcal{G}]^{\mathcal{H}}.$$ Since $R$ is normal,
$\mathcal{G}/\mathcal{H}$ is an affine group scheme with
$C[\mathcal{G}/\mathcal{H}]=C[\mathcal{G}]^{\mathcal{H}}$
(\cite{Dem}, III, Sec. 3, Thm. 5.6). Again by taking invariants the
isomorphism above restricts to an isomorphism
$$R^{\mathcal{H}} \otimes_L R^{\mathcal{H}} \simeq R^{\mathcal{H}}
\otimes_C C[\mathcal{G}/\mathcal{H}].$$

 The ring $R^{\mathcal{H}}$ is
$ID_q$-simple, because for every $ID_q$-ideal $P \subset
R^{\mathcal{H}}$, the ideal $P.R \subset R$ is an $ID_q$-ideal,
hence equals $(0)$ or $R$ and so $P=(P.R)^{\mathcal{H}}$ is $(0)$ or
$R^{\mathcal{H}}$. Since $L \subset R^{\mathcal{H}} \subset R$, we
also have $C(R^{\mathcal{H}})=C$. So by proposition
\ref{prop:caracpi}, $R^{\mathcal{H}}$ is an $ID_q$ Picard-Vessiot
ring  over $L$ with Galois group scheme $\mathcal{G}/\mathcal{H}$.
It remains to show that $E^{\mathcal{H}}=Loc(R^{\mathcal{H}})$.\\
Let $\tilde{L}:=Loc(R^{\mathcal{H}})$ and
$\tilde{\mathcal{G}}:=\underline{Gal}(E/\tilde{L})$. Then
$\mathcal{H}$ is a normal subgroup of $\tilde{\mathcal{G}}$ and by
the previous $(R.\tilde{L})^{\mathcal{H}}$ is a
$\tilde{\mathcal{G}}/\mathcal{H}$-torsor. But
$(R.\tilde{L})^{\mathcal{H}}=R^{\mathcal{H}}.\tilde{L}=\tilde{L}$,
so $\tilde{\mathcal{G}}=\mathcal{H }$, and hence
$E^{\mathcal{H}}=E^{\tilde{\mathcal{G}}}=\tilde{L}=Loc(R^{\mathcal{H}})$.
\item Without loss of generality we may assume that
$\mathcal{H}=\mathcal{G}$. Let us denote by $\mathcal{G}_{red}
 \subset \mathcal{G}$ the closed reduced subgroup given by the
nilradical ideal . Since $\mathcal{G}_{red}$ is normal in
$\mathcal{G}$, by the second statement
$\tilde{L}:=Loc(R^{\mathcal{G}_{red}})$ is an $ID_q$ Picard-Vessiot
extension of $L$ with Galois group scheme
$\underline{Gal}(\tilde{L}/L)=\mathcal{G}_{red}$. But this group
scheme is infinitesimal and so by \cite{Ch}, Cor. 1.12,
$\tilde{L}/L$ is purely inseparable. On the other hand, if $E/L$ is
inseparable and $p=char(L)$, then $ \tilde{L}:= E \cap
L^{\frac{1}{p}} \neq L$ is a finitely purely inseparable $ID_q$-ring
extension of $L$. Since every such extension is an $ID_q$-
Picard-Vessiot ring with an infinitesimal  Galois group scheme,
$\mathcal{G}$ has a non reduced quotient and therefore $\mathcal{G}$
is not reduced.
\end{enumerate}
\end{proof}
\subsubsection{Examples of  Galois groups}

\paragraph{The Galois group $\mathbb{G}_m$ in characteristic $p$}
~~\\
Let us
denote by $C=\overline{\mathbb{F}_p}$ the algebraic closure of
$\mathbb{F}_p$, where $p$ is a prime number. Let $F=C(t)$ be a
rational function field with coefficients in $C$. Let $(a_l)_{l \geq
0}$ be a set of elements in $\mathbb{F}_p$. Let $M =Fb_1$ be the
$ID_q$-module with corresponding $ID_qE$:
$$ \delta_M^{(np^k)}(y)= \frac{a_k}{ t^{np^k}}y$$ where $k \in
\mathbb{N}$ and
$$\delta_M^{(1)}(y)=\frac{y}{t}.$$

\begin{theorem}\label{theorem:compm}
Let $M$ be as above with its associated $ID_qE$, and let $\alpha
=\sum_{l\geq 0}a_lp^l \in \mathbb{Q}_p$. Then for an iterative
Picard-Vessiot extension $E/F$ for $M$, we have

$\underline{\rm{Gal}}(E/F) \simeq \mathbb{Z}/m\mathbb{Z}$ for some
$m$ if $\alpha \in \mathbb{Q}$ and $\underline{\rm{Gal}}(E/F) \simeq
\mathbb{G}_m$ if $\alpha \notin \mathbb{Q}$.

\end{theorem}
\begin{proof}
First of all, let us show that $\underline{\rm{Gal}}(E/F)$ is a
subgroup of $\mathbb{G}_m$. Let $y$ be a  solution of the $ID_qE$
associated to $M$, then $E=F(y)$.   Let $\tau \in
\underline{\rm{Gal}}(E/F)$ and $l \in \mathbb{N}$, we have
$\delta^{(np^l)}(\frac{\tau(y)}{y})=0$ and
$\delta^{(1)}(\frac{\tau(y)}{y})=0$. Thus, there exist $c \in C^*$
such that $\tau(y)=cy$.
Therefore, $\underline{\rm{Gal}}(E/F) \subseteq \mathbb{G}_m$.\\

Let us assume that $\alpha=\frac{a}{m}$ where $(a,m) \in
\mathbb{Z}\times \mathbb{N}^*$. Put $z=t^{a/m}$. Because
$z=t^{\alpha}$, we have $\delta^{(j)}(z)=0$ if $j\neq n^k$. We have
$$\delta^{(np^k)}(z^m)=\sum_{i_1+...+i_m=np^k}
\sigma_q^{i_2+...+i_m}(\delta^{(i_1)}(z))...\sigma_q^{i_m}(\delta^{(i_{m-1})}(z))\delta^{(i_m)}(z).$$
If one of the $i_j$ is not equal to $np^k$, there exists $i_l$ such
that $i_l \neq p^j$ for $j \leq np^k$. Then, an easy computation
shows that
 for all $k \in
\mathbb{N}$, $$\delta^{(np^k)}(z^m)=mz^{m-1}\delta^{(np^k)}(z).$$
 It follows that
$$mz^{m-1}\delta^{(np^k)}(z)=\binom{a}{np^k}_q t^{a-np^k}.$$
By Proposition \ref{prop:binform}, we have $\binom{a}{np^k}_q=ma_k$
and thus $\delta_M^{(np^k)}(z)= \frac{a_k}{ t^{np^k}}z$. Because
$E=F(z)$ and $z^m \in F$, we get that $\underline{\rm{Gal}}(E/F)$ is
a cyclic group.

Conversely, suppose that $y$ is an algebraic solution of the $ID_qE$
associated to $M$, then $E=F(y)$ is algebraic over $F$ and
$\underline{\rm{Gal}}(E/F)(C) \subsetneq \mathbb{G}_m(C)$ is a
cyclic group of order $m$. So there exist  $s \in \mathbb{Z}$ and
$(b_i)_{i\geq s}$ with $b_s=1$ such that $y^m= \sum_{i\geq s}b_it^i
\in F$. Thus,
$$my^{m-1}\delta^{(n)}(y)=y^m\frac{a_0}{t^n}=\delta^{(n)}(y^m)=\sum_{i \geq s}
b_i\binom{i}{n}_qt^{i-n}.$$ By comparing the coefficient of  $t^l$,
we obtain
$$  \ ma_o=b_i\binom{i}{n}_q \mbox{for all} \ i \geq s.$$ Since $b_s=1$ and
because of the properties of  $q$-binomials coefficients, we obtain
\begin{enumerate}
\item $s=k_sn$ with $k_s \in \mathbb{Z}$ and $a_0=\frac{k_s}{m}$,
\item $b_i=0$ for all $i \neq 0$ mod $n$.
\end{enumerate}

Induction using the higher iterative differences shows that $b_i=0$
for all $i>s$ and hence that $y^m=t^s$. By an argument used in the
first part of the proof it follows that $\alpha =\frac{s}{m}$.
\end{proof}
\paragraph{The Galois group $\mathbb{G}_m$ in characteristic $0$} Let
$L=\mathbb{C}(t)$ and let $q$ be a $n$-th primitive root of unity.
Let $M =Fb_1$ be a rank one $IDM_q(L)$-module and suppose that
$\Phi(b_1)=b_1$. Let $a \in \mathbb{C}$. Then, let us consider the
$ID_qE$ associated to $M$, that is $\delta^{(1)}(\bold{y})=0$ and
$\delta^{(n)}(\bold{y})=\frac{a}{nt^n}\bold{y}$.

\begin{theorem}\label{theorem:compmO} Let $M$ be as above with its
associated $ID_qE$. Then for an iterative Picard-Vessiot extension
$E/F$ for $M$, we have

$\underline{\rm{Gal}}(E/F)$ is finite cyclic if $a \in  \mathbb{Q}$
and $\underline{\rm{Gal}}(E/F) \simeq \mathbb{G}_m$ if $ a \notin
\mathbb{Q}$.

\end{theorem}
\begin{proof}
First of all, let us show that $\underline{\rm{Gal}}(E/F)$ is a
subgroup of $\mathbb{G}_m$. Let $y$ be a  solution of the $ID_qE$
associated to $M$, then $E=F(y)$. Let $\tau \in
\underline{\rm{Gal}}(E/F)$. Then, we have
\begin{enumerate}
\item $$\delta^{(1)}(\frac{\tau(y)}{y})=\sigma_q(\frac{1}{y})\tau(\delta^{(1)}y) +
\delta^{(1)}(\frac{1}{y})\tau(y)=0, \ (\delta^{(1)}(y)=0),$$
\item $$\delta^{(n)}(\frac{\tau(y)}{y})=(\frac{1}{y})\tau(\delta^{(n)}y) +
\delta^{(n)}(\frac{1}{y})\tau(y)=-\frac{a}{nt^n}\frac{\tau(y)}{y} +
\frac{1}{y} \tau(\frac{a}{nt^n}y)=0$$
\end{enumerate}
Thus, there exist $c \in C^*$ such that $\tau(y)=cy$.
Therefore, $\underline{\rm{Gal}}(E/F) \leq \mathbb{G}_m$.\\

Let us assume that $a=\frac{nb}{m}$ where $(b,m) \in
\mathbb{Z}\times \mathbb{N}^*$. Put $z=t^{nb/m}$. Because $z=t^{a}$,
we have $\delta^{(j)}(z)=0$ if $j\notin n\mathbb{N}$. We have
$$\delta^{(n)}(z^m)=\sum_{i_1+...+i_m=n}
\sigma_q^{i_2+...+i_m}(\delta^{(i_1)}(z))...\sigma_q^{i_m}(\delta^{(i_{m-1})}(z))\delta^{(i_m)}(z).$$
If one of the $i_j$ is not equal to $n$, there exists $i_l$ such
that $i_l \neq n$. Then, an easy computation  shows that

$$\delta^{(n)}(z^m)=mz^{m-1}\delta^{(n)}(z).$$
 It follows that,
$$mz^{m-1}\delta^{(n)}(z)=\binom{nb}{n}_q t^{nb-n}.$$
By Proposition \ref{prop:binform}, we have $\binom{nb}{n}_q= b = m
\frac{a}{n}$ and thus $\delta_M^{(n)}(z)=\frac{a}{nt^n} z$. Thus
$E=F(z)$ and $z^m \in F$. It follows  that
$\underline{\rm{Gal}}(E/F)$ is a finite cyclic group.

Conversely, suppose that $y$ is an algebraic solution of the $ID_qE$
associated to $M$, then $E=F(y)$ is algebraic over $F$ and
$\underline{\rm{Gal}}(E/F) \subsetneq \mathbb{G}_m$ is a cyclic
group of order $m$. So there exist  $s \in \mathbb{Z}$ and
$(b_i)_{i\geq s}$ with $b_s=1$ such that $y^m= \sum_{i\geq s}b_it^i
\in F$. Thus,
$$my^{m-1}\delta^{(n)}(y)=y^m\frac{a}{n t^n}=\delta^{(n)}(y^m)=\sum_{i \geq s}
b_i\binom{i}{n}_qt^{i-n}.$$ By comparing  the coefficient of  $t^l$,
we obtain that $  \frac{a}{n}= b_i\binom{i}{n}_q$ for all $i \geq
s$. Since $b_s=1$ and  because of  properties of the $q$-binomials
coefficients, we get that:
\begin{enumerate}
\item $s=k_sn$ with $k_s \in \mathbb{N}$ and $a=\frac{n k_s}{m}$;
\item $b_i=0$ for all $i \neq 0$ mod $n$.
\end{enumerate}

Induction using the higher iterative difference shows that $b_i=0$
for all $i>s$. It follows  that $y^m=t^s$ and  $a =\frac{n k_s}{m}$.
\end{proof}

\paragraph{The Galois group $\mathbb{G}_a$ in
 positive characteristic}

Let us denote by $C=\overline{\mathbb{F}_p}$ the algebraic closure
of $\mathbb{F}_p$, where $p$ is a prime number. Let $F=C(t)$ be a
rational function field with coefficients in $C$. Let $(a_l)_{l \geq
0}$ be a set of elements in $\mathbb{F}_p$. We choose $q \in C$ a
$n$-th primitive root of unity  with $n$ prime to $p$.\\

Let $M =Fb_1 \oplus F b_2$ be the $ID_q$-module with corresponding
$ID_qE$:
$$ \delta^{(np^k)}(Y)= A_k Y= \left( \begin{array}{cc}
0 & a_k \\
0 & 0 \end{array} \right) Y$$ for $k \in \mathbb{N}$.

\begin{theorem}\label{theorem:compa}
Let $M$ be as above with its associated $ID_qE$. Let $\alpha
=\sum_{l\geq 0}a_lp^l \in \mathbb{Q}_p$. Then for an iterative
Picard-Vessiot extension $E/F$ for $M$, we have

$\underline{\rm{Gal}}(E/F)$  is a  finite subgroup of order $r$ of
$\mathbb{G}_a$ if $\alpha \in \mathbb{Q}$ and
$\underline{\rm{Gal}}(E/F) \simeq \mathbb{G}_a$ if $\alpha \notin
\mathbb{Q}$.

\end{theorem}

For the proof, we need the following lemma.

\begin{lemma}\label{lemma:crit}
Let $(a_l)_{l \geq 0}$ be a sequence of elements in $\mathbb{F}_p$.
The following statements are equivalent :
\begin{enumerate}
\item The sequence  $(a_l)_{l \geq
0}$ is periodic from a certain rank;
\item $g=\sum_{l \in \mathbb{N}}a_lt^{np^l}  \in C((t))$ is
separable algebraic over $C(t)$.
\end{enumerate}
\end{lemma}
\begin{proof}
 see \cite{M} p.30 and replace $t$ by $t^n$.\\
\textbf{Proof of Theorem \ref{theorem:compa}}\\
We start with the iterative differential equation,

$$ \delta^{(np^k)}(Y)= A_k= \left( \begin{array}{cc}
0 & a_k \\
0 & 0 \end{array} \right) Y$$ for $k \in \mathbb{N}$.

Writing $Y=\left(
\begin{array}{c}y_1 \\
y_2 \end{array}\right)$, we find that $\delta^{(k)}(y_2)=0$ for all
$k \in \mathbb{N}$, which implies $y_2 \in C$. Using this result we
obtain $ \delta^{(np^k)}(y_1)=a_ky_2$ for all $k \in \mathbb{N} \
\mbox{and} \ \delta^{(1)}(y_1)=a_{-1}y_2.$ Thus, the formal solution
$y_1$ is equal to $$y_1 = y_2(\sum_{l \in \mathbb{N}}a_lt^{np^l}
).$$Then $E=F(y_1,y_2)=F(y_1)$, and for any $\tau \in
\underline{\rm{Gal}}(E/F)$ we get
$$\delta^{(np^l)}(\tau(y_1)-y_1)=\tau(\delta^{(np^l)}(y_1))-\delta^{(np^l)}(y_1)=\tau(y_2a_l)-y_2a_l=0.$$
thus there exists $d \in c$ such that $\tau(y_1)=y_1 +d$. Therefore
$\underline{\rm{Gal}}(E/F)$ is a subgroup of $ \mathbb{G}_a$.

 Using  Lemma \ref{lemma:crit}, we obtain
 \begin{enumerate}
\item the solution $y_1$ is separable algebraic over $F$ if $\alpha \in
\mathbb{Q}$ (the sequence  $(a_l)_{l \geq 0}$ is periodic from a
certain index if and only if  $\alpha \in \mathbb{Q}$), so the
Galois group is actually finite.
\item If $\alpha \notin \mathbb{Q}$, then $y_1$ is transcendent
over $F$, and hence $E/F$ is purely transcendental of degree $1$,
showing that $\underline{\rm{Gal}}(E/F) \simeq \mathbb{G}_a$.
\end{enumerate}
\end{proof}
\begin{remark}
These examples of iterative $q$-difference equations are obtained by
$q$-deformation of the examples of B.H. Matzat in \cite{M} example
$2.14$ and $2.15$. The Galois groups obtained here are the same as
those obtained by Matzat. The fact that simple Galois groups such as
$\bold{G}_m$ and $\bold{G}_a$ do not degenerate by $q$-deformation
give us a nice hope for confluence studies.
\end{remark}

\section{An analogue of the Grothendieck-Katz conjecture}

In this section, we state an analogue of the Grothendieck-Katz
conjecture for iterative $q$-difference equations. In \cite{Lu}, L.
Di Vizio proves this conjecture for $q$-difference equations with
$q$ non equal to a root of unity and algebraic over $\mathbb{Q}$.
Briefly, she shows that given a $q$-difference equation,
$\mathcal{L}y=0$ with coefficients in $\mathbb{Q}(t)$, one can
describe the behavior of the solutions of $\mathcal{L}$ by
considering the reduction of $\mathcal{L}$ modulo
the prime numbers.\\
\begin{Not}
Let $K$ be a number field and $\mathcal{O}_K$ the ring of integers
of $K$. Let $q\in K^*$ a $n$-th root of unity. We denote by
$\Sigma_f$ the set of all finite places $v$ of $K$. The uniformizer
of the finite place $v$ is denoted by $\pi_v$ and $| . |_v$ denotes
the $v$-adic absolute value of $K$. We denote by $p_v$ be the
characteristic of the residue field $k_v$ of $\pi_v$.\end{Not}

 Let $(\mathcal{M}, \phi_M, \delta_M^*)$ be an iterative $q$-difference
module defined over $K(t)$. By point $5$ of Definition
\ref{defin:mod}, we get that
$(\delta_M^{(1)})^n=[n]_q!\delta_M^{(n)}=0$. By Proposition $2.1.2$
of \cite{Lu}, this implies that $\phi_M^n=id_M$ and that $M$ is
trivial as ordinary $q$-difference module. In that case, the $q$-analogous of
the Grothendieck-Katz conjecture of L. Di Vizio (see theorem $7.1.1$ of \cite{Lu}) is  trivially
satisfied.\\

This fact already appears in the work of Matzat-van der Put on
iterative differential equations (see \cite{MP} Remarks p.51). If one considers only the first
derivation $\partial_M^{(1)}$ of an iterative differential module $M$, it is a nilpotent operator of order $p$, the characteristic of the base field, and the iterative differential
module $M$ is always trivial regarded as differential module in the classical sense. This
observation emphasizes the fact that one has to consider the
operator, iterative difference or derivations of higher order, and
not simply its first rank  to characterize the behavior of the iterative module.\\

In our case, it means that all the information is encompassed in the
iterative $q$-difference of order $n$, i.e. $\delta_M^{(n)}$. If we
 change the basis of $M$  so that the action of $\phi_M$ on this new
 basis  is given by the identity map, one has

 $$\delta_M^{(n)}(\lambda e)= \delta^{(n)}(\lambda) e+ \lambda
 \delta_M^{(n)}(e) \ \mbox{for all} \ \lambda \in K(t), e \in
 \mathcal{B}.$$
 That is the operator $\delta_M^{(n)}$ behave quite like a
 connection. For differential equations over a zero characteristic base field,
  the Grothendieck conjecture can be restate in terms of $p$-curvatures, i.e.,\\
 \textit{A differential equation $Ly=0$ with $L \in \mathbb{Q}[\partial]$
has a full set of algebraic solutions if and only if for almost all
primes $p \in \mathbb{Z}$ the reduction modulo $p$ of $Ly=0$ has a
full set of solutions in  $\bold{F}_p(t)$ i.e. the $p$-curvature of
$L$, i.e. the $p$-iterate of the connection of the differential equation is equal to zero.}\\

  In analogous with this case, one may introduce the following definition.
\begin{defin}
Let  $(\mathcal{M}, \phi_M, \delta_M^*)$ be an iterative
$q$-difference module defined over $K(t)$. One defines  the $\pi_v$-curvature $\psi_{\pi_v}$ of $M$ as
the $p_v$-iterate, in the sense of the composition, of the operator $\delta_M^{(n)}$ i.e.
 $$\psi_{\pi_v}:=(\delta_M^{(n)})^{p_v}.$$
\end{defin}
Now, we are able to state our analogous of the Grothendieck-Katz conjecture for the iterative $q$-difference modules.

\begin{Conj}\label{Conj:luci}
Let  $(\mathcal{M}, \phi_M, \delta_M^*)$ be an iterative
$q$-difference module defined over $K(t)$. The iterative
$q$-difference module
 $\mathcal{M}$  is isotrivial, i.e. becomes trivial after a finite base field extension if and only if
 for almost all finite places $v$,
the $\pi_v$-curvature $\psi_{\pi_v}$ induces the zero map on the
reduction of $\mathcal{M}$ modulo $\pi_v$.
\end{Conj}

\textbf{Computation of the curvature}\\

Fix a basis $\underline{e}$ of  $M$ such that the actions of $\delta_M^{(1)}, \delta_M^{(n)}$ w.r.t. $\underline{e}$ are  given by
$$\delta^{(1)}\underline{e}=0 \ \mbox{and} \ \delta^{(n)}(\underline{e})=A\underline{e} \ \mbox{with} \ A \in M_r(K(t)).$$
Set $A_{[1]}:= A$ and define inductively $A_{[k]}$ with $A_{[k+1]}:= \delta^{(n)}(A_{[k]}) +A_{[k]} A_{[1]}$. Then,
$$ (\delta^{(n)})^k(\underline{e})=A_{[k]}\underline{e}.$$
The matrix of the $\pi_v$-curvature $\psi_{\pi_v}$ with respect to the basis $\underline{e}$ is $A_{[p_v]}$.

Here is an example where Conjecture \ref{Conj:luci} holds.

\begin{ex}[Example \ref{ex:car0}]
 Let $a \in K$. Then, let us consider the
$ID_qE$ : $$\delta^{(1)}(\bold{y})=0 \ \mbox{ and} \
\delta^{(n)}(\bold{y})=\frac{a}{t^n}\bold{y}.$$
 Let $v$ be a sufficiently large place
of $K$. A simple calculation shows that the reduction of $
(\delta_M^{(n)})^{p_v}$ modulo $\pi_{v}$ is equal to
$\frac{a(a-1)...(a-(p_v-1))}{t^{np_v}}id_M$ (we have $A_{[1]}=\frac{a}{t^n},A_{[2]}=\frac{-a}{t^{2n}}
+\frac{a^2}{t^n}=\frac{a(a-1)}{t^{2n}},...$).\\
 If we assume that for almost all finite
places $v$, the reduction modulo $\pi_{v}$  of $(\delta_M^{(n)})^{p_v}$ is
equal to zero, we get that for almost all finite places $v$ there
exists $a_v \in \mathbb{Z}$ such that the valuation of  $a-a_v$ in
$\pi_v$ is strictly positive. By the Density Theorem of Chebotarev,
we obtain that $a \in \mathbb{Q}$. We have proved in Theorem
\ref{theorem:compmO} that $a \in \mathbb{Q}$ if and only if
 $\mathcal{M}$ has a finite Galois group.
\end{ex}

It would be also interesting to relate isotrivial $q$-difference module over $K(t)$ (in the classical  sense) and iterative
$q$-difference module. If one consider an element $q \in K$ not a root of unity, its reduction $q_v$ at a finite place $v$, if it exists, is a root
of unity. Thus, starting from a $q$-difference module one could ask what are the conditions such that given a finite place
$v$ of $K$ the reduction of $M$ modulo $\pi_v$ can be endowed with a structure of iterative $q_v$-difference module.\\

For iterative differential modules, this question give rise to a conjecture enounced by Matzat and Van
der Put (\cite{MP} p.51). The analogue of this conjecture in the $q$-difference world is \\
\\ \textit{Let a $q$-difference module $M$ over $K(t)$. Suppose that for almost all finite places $v$,
 the reduction of $\mathcal{M}$ modulo $\pi_v$ has a structure of iterative $q_v$-difference module and has a finite Galois
 group $G$ then the difference Galois group of $M$ is isomorphic to $G$.}\\
 
 This statement should be  a consequence or reformulation of the theorem of L. Di Vizio, that is Theorem $7.1.1$ in
 \cite{Lu}.

\end{document}